\documentclass[reqno]{amsart}
\usepackage{amssymb,latexsym,amscd,hyperref,tikz-cd,bbm} 
\usepackage{amsmath}
\usepackage{pst-node}
\usepackage{xcolor}
\makeatletter
\renewcommand{\pod}[1]{\allowbreak\mathchoice
  {\if@display \mkern 18mu\else \mkern 8mu\fi (#1)}
  {\if@display \mkern 18mu\else \mkern 8mu\fi (#1)}
  {\mkern4mu(#1)}
  {\mkern4mu(#1)}
}
\usepackage{enumerate}[1.)]
\usepackage{pb-diagram}  
\renewcommand{\eqref}[1]{(\ref{#1})}   
\theoremstyle{plain}

\numberwithin{equation}{section}

\newtheorem{theorem}{Theorem}[section]
\newtheorem{theorema}{Theorem}

\newtheorem{lemma}[theorem]{Lemma}

\newtheorem{proposition}[theorem]{Proposition}
\newtheorem{algorithm}{Algorithm}

\newtheorem{definition}[theorem]{Definition}

\newtheorem{remark}[theorem]{Remark}

\newcommand {\Q}{{\mathbb{Q}}}
\newcommand {\C}{{\mathbb{C}}}

\newcommand {\Z}{{\mathbb{Z}}}
\newcommand {\N}{{\mathbb{N}}}
\newcommand {\F}{{\mathbb{F}}}

\newcommand {\PP}{\mathbb{P}}

\newcommand {\LL}{{\mathcal{L}}}

\newcommand{\Gal}      {\mathop{\rm {Gal}}}

\newcommand{\Cl}      {\mathop{\rm {Cl}}}

\hypersetup{
    unicode=false,        
    pdftoolbar=true,      
    pdfmenubar=true,       
    pdffitwindow=false,     
    pdfstartview={FitH},    
    pdftitle={},    
    pdfauthor={},     
    colorlinks=true,       
   linkcolor=blue,          
    citecolor=blue,        
    filecolor=blue,      
    urlcolor=blue}
\begin{document}

\title{Malle's conjecture for nonic Heisenberg extensions}
\date{\today}
\author{\'Etienne Fouvry}
\address{Universit\' e Paris--Saclay, CNRS, Laboratoire de math\' ematiques d'Orsay, 91405 Orsay, France}
\email{Etienne.Fouvry@universite-paris-saclay.fr}
\author{Peter Koymans}
\address{Max Planck Institute for Mathematics, Vivatsgasse 7, 53111 Bonn, Germany}
\email{koymans@mpim-bonn.mpg.de}
 

\begin{abstract}   
We prove Malle's conjecture for nonic Heisenberg extensions over $\Q$. Our main algebraic result shows that the number of nonic Heisenberg extensions over $\Q$ with discriminant bounded by $X$ is given by a character sum. We then extract the main term from this sum by exploiting oscillation of characters.
\end{abstract}

\maketitle

\section{Introduction}
\label{intro}
A fundamental problem in arithmetic statistics is to count algebraic extensions over $\Q$ with bounded discriminant. This subject has its roots in a famous theorem due to Hermite that there are only finitely many number fields with bounded discriminant.

Let $K/\Q$ be an extension of degree $n$ and write $L$ for the normal closure of $K$. Then $\Gal(L/\Q)$ acts on the $n$ embeddings $K \hookrightarrow \overline{\Q}$, which gives a homomorphism from $\Gal(L/\Q)$ to $S_n$. By abuse of notation we define $\Gal(K/\Q) \subseteq S_n$ to be the image of this homomorphism. We then define for every transitive group $G \subseteq S_n$ the counting function
\[
N(G, X) := |\{K/\Q : \Gal(K/\Q) \cong G, \Delta_{K/\Q} \leq X\}|,
\]
where $\Delta_{K/\Q}$ is the absolute discriminant and the fields $K$ are taken inside a fixed algebraic closure of $\Q$. Here we stress that the isomorphism is not just an isomorphism of finite groups but as permutation groups; this is equivalent to $G$ and $\Gal(K/\Q)$ being conjugate subgroups of $S_n$. 
This counting function is the subject of Malle's conjecture \cite{Malle, Malle2}, who conjectured an asymptotic of the form
\begin{align}
\label{eStrongMalle}
N(G, X) \sim c(G) X^{a(G)} (\log X)^{b(G) - 1},
\end{align}
where $c(G)$ is an unspecified constant and where $a(G)$ and $b(G)$ can be computed as follows. Let $G \subseteq S_n$, so that $G$ has a natural action on the set $\{1, \dots, n\}$. Then put for $\sigma \in G$
\[
\text{ind}(\sigma) := n - |\{\text{orbits of } \sigma\}|,
\]
where the orbits are with respect to the action on $\{1, \dots, n\}$. We define
\[
a(G)^{-1} := \min_{\sigma \in G \setminus \{\text{id}\}} \text{ind}(\sigma).
\]
To define $b(G)$, we consider the following action of $\Gal(\overline{\mathbb{Q}}/\mathbb{Q})$ on $G$. Let $c: \Gal(\overline{\mathbb{Q}}/\mathbb{Q}) \rightarrow \hat{\mathbb{Z}}^\ast$ be the cyclotomic character. For $g \in G$ and $\sigma \in \Gal(\overline{\mathbb{Q}}/\mathbb{Q})$, we define
\[
g^\sigma := g^{c(\sigma)}.
\]
It is easy to see that this induces an action of $\Gal(\overline{\mathbb{Q}}/\mathbb{Q})$ on $C(G)$, the conjugacy classes of $G$. We remark that $\text{ind}(\sigma)$ is constant as $\sigma$ varies through a conjugacy class $C$, which allows us to define $\text{ind}(C)$ in the obvious way. Furthermore, the index of $g$ is the same as the index of $g^\sigma$. Then we define
\[
b(G) := |\{C \in C(G) : \text{ind}(C) = a(G)^{-1}\}/\sim|,
\]
where two conjugacy classes are equivalent if they are in the same orbit under the action of $\Gal(\overline{\mathbb{Q}}/\mathbb{Q})$ on $C(G)$. As stated the exponent $b(G)$ in Malle's conjecture is not always correct, see the work of Kl\"uners  \cite{Klunerscounter} for a counterexample. T\"urkelli \cite{Turkelli} proposed a modified version of Malle's conjecture, with a different $b(G)$, to take into account the counterexample found by Kl\"uners.

Equation (\ref{eStrongMalle}) is known in a limited number of cases, see the work of Wright \cite{Wright} for abelian $G$, Davenport--Heilbronn \cite{DH} for $S_3$, Kl\"uners \cite{Kluners} for generalized quaternion groups, Bhargava \cite{Bhargava1, Bhargava2} for $S_4$ and $S_5$, Bhargava--Wood \cite{BW} for $S_3 \subseteq S_6$ and \cite{Wang1, Wang2} for direct products $G \times A$ with $G \in \{S_3, S_4, S_5\}$ and $A$ abelian. Alberts \cite{Alberts1, Alberts2} made progress for many solvable groups. Finally, equation (\ref{eStrongMalle}) is also known for quartic $D_4$-extensions, see the work \cite{CDO} that we reproduce now.

\begin{theorema}[Cohen--Diaz y Diaz--Olivier]
\label{tDisc}
The number of degree $4$ extensions $L$ of $\Q$, up to isomorphism, such that the normal closure has Galois group isomorphic to $D_4$, with absolute discriminant at most $X$ is asymptotic to $c(D_4) X$, where
\[
c(D_4) = \frac{3}{\pi^2} \sum_D \frac{2^{-i(D)}}{D^2} \frac{L(1, D)}{L(2, D)}.
\]
Here the sum is over fundamental discriminants different from $1$, and $i(D) = 0$ if $D > 0$ and $i(D) = 1$ if $D < 0$.
\end{theorema}

The error term in Theorem \ref{tDisc} is of exceptional quality, namely of size $O_\epsilon(X^{\frac{3}{4} + \epsilon})$. In this paper we are interested in nonic Heisenberg extensions, which bear some similarities with quartic $D_4$-extensions. Let $\text{Heis}_3$ be the Heisenberg group with $27$ elements, i.e. the multiplicative group of upper triangular matrices with coefficients in $\mathbb{F}_3$ and ones on the diagonal. Denote by $N(\text{Heis}_3, X)$ the number of degree $9$ extensions $L$ of $\Q$, up to isomorphism, such that the normal closure has Galois group isomorphic to $\textup{Heis}_3$ and such that the absolute discriminant is bounded by $X$. Our main result is the following.

\begin{theorem}
\label{tMain}
There exists a constant $c({\rm Heis}_3) > 0$ such that
\[
N(\textup{Heis}_3, X) \sim c({\rm Heis}_3) \, X^{1/4}.
\]
\end{theorem}

We give a completely explicit formula for $c({\rm Heis}_3)$, which we postpone until equation (\ref{defcoeff}). In Remark \ref{rComparison}, we will compare the constants $c({\rm Heis}_3)$ and $c (D_4)$. Actually, our proof leads to the asymptotic formula
\[
N(\text{Heis}_3, X) = c({\rm Heis}_3) \, X^{1/4} + O_A(X^{1/4} (\log X)^{-A})
\]
for all $A > 0$.

Our main theorem implies Malle's conjecture for nonic Heisenberg extensions (note that, up to conjugation, there is precisely one transitive subgroup of $S_9$ isomorphic to $\text{Heis}_3$). One of the challenges is to find an explicit expression for the constant $c({\rm Heis}_3)$. Indeed, it is substantially easier to show that there exists a constant $c({\rm Heis}_3)$. This phenomenon can already be observed in the work of \cite{CDO2}, where the strong form of Malle's conjecture is proved, with an explicit constant $c(G, K)$, for cyclic degree $\ell$ extensions over an arbitrary base field $K$.

Despite the superficial similarities between Theorem \ref{tDisc} and Theorem \ref{tMain}, the proof techniques employed in Theorem \ref{tDisc} break down completely for nonic Heisenberg extensions. The key principle used in the proof of Theorem \ref{tDisc} is the following: take a quadratic extension $K/\Q$ and take a quadratic extension $L/K$. Then typically $L$ is a quartic $D_4$-extension of $\Q$. The problem then reduces to uniformly counting quadratic extensions.

However, this does not seem to be true for cyclic degree $\ell$ extensions. Instead we take an entirely different approach, where we estimate a certain character sum that counts the number of Heisenberg extensions. Our approach is in spirit of the work of Heath-Brown \cite{HB} and Fouvry--Kl\"uners \cite{Fo--Kl1, Fo--Kl2, Fo--Kl4, Fo--Kl3}, although the technical details are somewhat different than these works.

We believe that Theorem \ref{tMain} can be extended in various directions. As a first generalization one can consider the Heisenberg group $\text{Heis}_\ell$ of order $\ell^3$, where $\ell \geq 3$ is a prime. Our algebraic results are in fact stated in this more general setting. However our analytic results currently use that $\Z[\zeta_3]$ is a principal ideal domain. It is possible to extend our analytic results to any odd prime $\ell$ for which $\Z[\zeta_\ell]$ is a principal ideal domain (so $\ell \in \{3, 5, 7, 11, 13, 17, 19\}$), and perhaps it is possible to extend them to all odd primes $\ell$.

Another direction to consider is to count Heisenberg extensions in the regular representation. The resulting counting function has some similarities to the ones considered in Fouvry--Luca--Pappalardi--Shparlinski \cite{FLPS} and Klys \cite{Klys}. We are optimistic that our techniques also apply here. A final direction that we shall discuss in this introduction is to count extensions by conductor instead of discriminant. This was done in \cite{Conductor} for quartic $D_4$-extensions. Perhaps it is possible to extend our results to this setting as well.

\section*{Acknowledgements}
We thank Carlo Pagano for several inspiring conversations that led to the proof of Theorem \ref{tHeisenberg}. Peter Koymans wishes to thank the Max Planck Institute for Mathematics in Bonn for its financial support, great work conditions and an inspiring atmosphere.

\section{The Heisenberg group}
In this section we develop the algebraic theory for the Heisenberg group ${\rm Heis}_\ell$ of order $\ell^3$ with $\ell \geq 3$ a prime. We start by fixing an algebraic closure $\overline{\Q}$ once and for all. We also fix algebraic closures $\overline{\Q_p}$ for all prime numbers $p$. All our number fields and local fields are implicitly taken inside these fixed algebraic closures. All our cohomology groups have to be interpreted as profinite group cohomology.

\subsection{The different ideal}
For a local or global field $K$, we write $\mathcal{O}_K$ for its ring of integers. If $L/K$ is an extension of local or global fields, we write $\mathfrak{d}_{L/K}$ for the different ideal and $\Delta_{L/K}$ for the relative discriminant. Recall that $\mathfrak{d}_{L/K}$ is an ideal of $L$, while $\Delta_{L/K}$ is an ideal of $K$. Denote by $f_\alpha$ the minimal polynomial of an element $\alpha$ and denote by $e_{\mathfrak{q}/\mathfrak{p}}$ the ramification index of the prime $\mathfrak{q}$ of $L$ lying above a prime $\mathfrak{p}$ of $K$. We now record the following well-known properties of the different ideal.

\begin{lemma}
\label{lDifferent}
Let $L/K$ be an extension of local or global fields. Let $\mathfrak{q}$ be a prime of $L$ and let $\mathfrak{p}$ be the prime of $K$ below $\mathfrak{q}$. The different ideal satisfies the following properties
\begin{enumerate}
\item[(i)] we have ${\rm N}_{L/K}(\mathfrak{d}_{L/K}) = \Delta_{L/K}$;
\item[(ii)] we have $\mathfrak{d}_{M/L} \mathfrak{d}_{L/K} = \mathfrak{d}_{M/K}$;
\item[(iii)] we have $\mathfrak{q} \mid \mathfrak{d}_{L/K}$ if and only if $\mathfrak{q}$ is ramified in $L/K$. Furthermore, in case that $\mathfrak{q}$ is not wildly ramified, we have that $\mathfrak{q}^{e_{\mathfrak{q}/\mathfrak{p}} - 1}$ exactly divides $\mathfrak{d}_{L/K}$;
\item[(iv)] we have
\[
v_\mathfrak{q}(\mathfrak{d}_{L/K}) = v_\mathfrak{q}(\mathfrak{d}_{L_\mathfrak{q}/K_\mathfrak{p}});
\]
\item[(v)] $\mathfrak{d}_{L/K}$ is generated by the elements $f'_\alpha(\alpha)$ as $\alpha$ ranges over all elements of $\mathcal{O}_L$ such that $L = K(\alpha)$. Now suppose additionally that $\alpha$ is an element of $L$ such that $\mathcal{O}_L = \mathcal{O}_K[\alpha]$. Then $\mathfrak{d}_{L/K} = (f'_\alpha(\alpha))$.
\end{enumerate}
\end{lemma}

Our next result is known in the case $k = \Q$, but we were unable to find a reference for general $k$.

\begin{lemma}
\label{lBicyclic}
Let $\ell$ be a prime number. Suppose that $K/k$ is an extension of local or global fields such that $\Gal(K/k) \cong \mathbb{F}_\ell^2$. Write $k_1, \dots, k_{\ell + 1}$ for the intermediate fields. Then we have
\[
\Delta_{K/k} = \prod_{i = 1}^{\ell + 1} \Delta_{k_i/k}.
\]
\end{lemma}

\begin{proof}
Let for now $K/k$ be any finite Galois extension of local or global fields. The conductor--discriminant formula states that
\begin{align}
\label{eCD}
\mathfrak{d}_{K/k} = \prod_{\chi \in \text{Irr}(G)} \mathfrak{f}(\chi)^{\chi(1)},
\end{align}
where $\text{Irr}(G)$ denotes the set of irreducible characters of $G = \Gal(K/k)$ and $\mathfrak{f}(\chi)$ denotes the Artin conductor of $\chi$, see \cite[Chapter VI]{Serre} for the definition of the Artin conductor.

If $K/k$ is bicyclic, then there are $\ell^2$ irreducible characters. Except for the trivial character, there are $\ell - 1$ non-trivial characters coming from each $\Gal(k_i/k)$ for $i = 1, \dots, \ell + 1$. Choose one non-trivial character $\chi_i$ for $\Gal(k_i/k)$, so that all non-trivial characters for $\Gal(k_i/k)$ are $\chi_i^j$ for $j = 1, \dots, \ell - 1$. It is also proven in \cite[Chapter VI, Proposition 6]{Serre} that the Artin conductor of $\chi_i^j$ is the same as the Artin conductor of $\chi_i^j$ restricted to $\Gal(k_i/k)$. We conclude that
\[
\mathfrak{d}_{K/k} = \prod_{i = 1}^{\ell + 1} \prod_{j = 1}^{\ell - 1} \mathfrak{f}(\chi_i^j) = \prod_{i = 1}^{\ell + 1} \mathfrak{d}_{k_i/k}
\]
by two applications of equation (\ref{eCD}). The lemma follows once we take norms.
\end{proof}

\subsection{General theory}
Let $\ell$ be an odd prime. The Heisenberg group $\text{Heis}_\ell$ is the multiplicative group of upper triangular matrices with coefficients in $\mathbb{F}_\ell$ (and ones on the diagonal). $\text{Heis}_\ell$ is a non-commutative group of size $\ell^3$ with center $Z(\text{Heis}_\ell)$ of size $\ell$. The quotient $\text{Heis}_\ell/Z(\text{Heis}_\ell)$ is bicyclic so that $\text{Heis}_\ell$ is a central $\mathbb{F}_\ell$-extension of $\mathbb{F}_\ell^2$. Furthermore, every element has order $\ell$.

Recall that the central extensions of $\mathbb{F}_\ell^2$ by $\mathbb{F}_\ell$ are parametrized by the group $H^2(\mathbb{F}_\ell^2, \mathbb{F}_\ell)$, where we view $\mathbb{F}_\ell$ as a trivial $\mathbb{F}_\ell^2$-module. Write $\chi_1$ and $\chi_2$ for the two natural projection maps from $\mathbb{F}_\ell^2$ to $\mathbb{F}_\ell$. Then it is shown in \cite[Section 4.1]{Ko--Pa} that the Heisenberg group is precisely realized by the $1$-dimensional subspace of $H^2(\mathbb{F}_\ell^2, \mathbb{F}_\ell)$ generated by the $2$-cocycle $(\sigma, \tau) \mapsto \chi_1(\sigma) \chi_2(\tau)$, which we denote by $\theta_{\chi_1, \chi_2}(\sigma, \tau)$.

The inflation--restriction exact sequence will play an important role throughout this section. Let $G$ be a profinite group, $N$ a normal open subgroup and $A$ a discrete $G$-module. Then the quotient $G/N$ naturally acts on the fixed points $A^N$. We have a long exact sequence
\begin{multline}
\label{eInfRes}
0 \rightarrow H^1(G/N, A^N) \xrightarrow{\text{inf}} H^1(G, A) \xrightarrow{\text{res}} H^1(N, A)^{G/N} \\ 
\xrightarrow{\text{tr}} H^2(G/N, A^N) \xrightarrow{\text{inf}} H^2(G, A).
\end{multline}
Here the map $\text{tr}$ is known as the \emph{transgression} map, while the other maps are the usual \emph{inflation} and \emph{restriction} maps. We remark that $G/N$ naturally acts on $H^1(N, A)$ by sending a cocycle $f: N \rightarrow A$ to $(g * f)(n) = g * f(g^{-1}ng)$.

Over number fields the Heisenberg group is realized as follows. Take two linearly independent characters $\chi, \chi': G_\Q \rightarrow \mathbb{F}_\ell$ and let $K$ be the bicyclic extension given by $\chi$ and $\chi'$. We apply equation (\ref{eInfRes}) with $A = \mathbb{F}_\ell$, $G = G_\Q$ and $N = G_K$, where $G_L$ denotes the absolute Galois group of a field $L$. Here, and for the remainder of this paper, we view $\mathbb{F}_\ell$ as a discrete Galois module with trivial action. In this case we get an isomorphism
\begin{align}
\label{eInfResCor}
\frac{\text{Hom}(G_K, \mathbb{F}_\ell)^{\Gal(K/\Q)}}{\text{Hom}(G_\Q, \mathbb{F}_\ell)} \cong \text{ker}(H^2(\Gal(K/\Q), \mathbb{F}_\ell) \xrightarrow{\text{inf}} H^2(G_\Q, \mathbb{F}_\ell)).
\end{align}
If $K$ is a field and $\chi: G_K \rightarrow \mathbb{F}_\ell$ is a character, we write $K(\chi)$ for the field extension of $K$ corresponding to $\chi$. The space $\text{Hom}(G_K, \mathbb{F}_\ell)^{\Gal(K/\Q)}$ consists of those characters $\chi \in \text{Hom}(G_K, \mathbb{F}_\ell)$ satisfying the following two properties. Firstly, $K(\chi)/\Q$ is a Galois extension. Secondly, there is an exact sequence
\begin{align}
\label{eExt}
1 \rightarrow \Gal(K(\chi)/K) \rightarrow \Gal(K(\chi)/\Q) \rightarrow \Gal(K/\Q) \rightarrow 1
\end{align}
with $\Gal(K(\chi)/K)$ central in $\Gal(K(\chi)/\Q)$. As explained in \cite[Section 4]{Ko--Pa}, the isomorphism in equation (\ref{eInfResCor}) is then explicitly given as follows. Using $\chi$ to identify $\Gal(K(\chi)/K)$ with $\mathbb{F}_\ell$ in equation (\ref{eExt}), we naturally get a class in the second cohomology group $H^2(\Gal(K/\Q), \mathbb{F}_\ell)$.

We conclude that if $\theta_{\chi, \chi'}(\sigma, \tau)$ is trivial in $H^2(G_\Q, \mathbb{F}_\ell)$, there exists an extension $M/\Q$ containing $\Q(\chi, \chi')$ with $\Gal(M/\Q) \cong \text{Heis}_\ell$. Conversely, if there exists such an extension $M/\Q$, then $\theta_{\chi, \chi'}(\sigma, \tau)$ is trivial in $H^2(G_\Q, \mathbb{F}_\ell)$. 

\begin{definition}
For an extension $K/\Q$ with $\Gal(K/\Q) \cong \mathbb{F}_\ell^2$, we define $\textup{Heis}(K/\Q)$ to be the subspace of $\rho \in \textup{Hom}(G_K, \mathbb{F}_\ell)^{\Gal(K/\Q)}$ that maps to the $1$-dimensional subspace of $H^2(\mathbb{F}_\ell^2, \mathbb{F}_\ell)$ generated by the $2$-cocycles $\theta_{\chi_1, \chi_2}(\sigma, \tau)$ under the transgression map. If $\rho \in \textup{Heis}(K/\Q)$ and $\chi: G_\Q \rightarrow \mathbb{F}_\ell$, we call $\rho + \chi \in \textup{Heis}(K/\Q)$ the twist of $\rho$ by $\chi$.
\end{definition}

\begin{remark}
\label{rSelfcup}
The transgression map naturally lands in $H^2(\Gal(K/\Q), \mathbb{F}_\ell)$, not in $H^2(\mathbb{F}_\ell^2, \mathbb{F}_\ell)$. Hence we are implicitly choosing an isomorphism $\Gal(K/\Q) \cong \mathbb{F}_\ell^2$ in the above definition, which allows us to identify 
\[
H^2(\Gal(K/\Q), \mathbb{F}_\ell) \cong H^2(\mathbb{F}_\ell^2, \mathbb{F}_\ell).
\]
Take any character $\chi: \mathbb{F}_\ell^2 \rightarrow \mathbb{F}_\ell$. Observe that the $2$-cocycle $\theta_{\chi, \chi}(\sigma, \tau) \in H^2(\mathbb{F}_\ell^2, \mathbb{F}_\ell)$ is trivialized by the $1$-cochain that sends $\sigma$ to $\chi(\sigma)^2/2$. Using this, we directly verify that the choice of isomorphism does not change the set $\textup{Heis}(K/\Q)$.
\end{remark}

Our final lemma gives a convenient way to decide if two degree $\ell^2$ Heisenberg extensions of $\Q$ are isomorphic.

\begin{lemma}
\label{lHIsoTest}
Let $\ell$ be an odd prime number. Let $L$ and $L'$ be two degree $\ell^2$ extensions of $\Q$ such that the Galois groups of the normal closures $N(L)$ and $N(L')$ are isomorphic to the Heisenberg group ${\rm Heis}_\ell$. Then $L$ and $L'$ are isomorphic if and only if $N(L)$ is isomorphic to $N(L')$ and $L$ and $L'$ contain the same degree $\ell$ subfield.
\end{lemma}

\begin{proof}
Certainly, if $L$ and $L'$ are isomorphic, then $N(L)$ and $N(L')$ are isomorphic. Furthermore, since the degree $\ell$ subfield is Galois over $\Q$, they must be the same. 

Reversely, suppose that $L$ and $L'$ are as in the lemma. By Galois theory, $L$ and $L'$ correspond to non-normal subgroups $H$ and $H'$ of the Heisenberg group $\text{Heis}_\ell$ of order $\ell$. Then, since $L$ and $L'$ contain the same degree $\ell$ subfield, it follows that $H$ and $H'$ together generate a subgroup of order $\ell^2$. From the structure of the Heisenberg group, we see that $H$ and $H'$ are then conjugate in $\text{Heis}_\ell$. This implies that $L$ and $L'$ are isomorphic.
\end{proof}

\subsection{\texorpdfstring{Heisenberg extensions of $\Q_\ell$}{Heisenbergextensions of Ql}}
Let us first analyze the situation locally at $\ell$. Since every element of $\text{Heis}_\ell$ has order $\ell$, its ramification theory is relatively simple. We further profit from the fact that $\Q_\ell$ has only two linearly independent cyclic degree $\ell$ extensions unlike $\Q_2$.

\begin{lemma}
\label{lMichailov}
Let $K$ be any field of characteristic $0$ containing a primitive $\ell$-th root of unity $\zeta_\ell$. For $\alpha \in K^\ast$, we write $\chi_\alpha$ for a character corresponding to $K(\sqrt[\ell]{\alpha})$. Then $\theta_{\chi_\alpha, \chi_\beta}$ is trivial in $H^2(G_K, \mathbb{F}_\ell)$ if and only if there exists $\omega \in K(\sqrt[\ell]{\alpha})$ such that ${\rm N}_{K(\sqrt[\ell]{\alpha})/K}(\omega) = \beta$. In this case the Heisenberg extension can be obtained by adjoining the $\ell$-th root of the element
\[
\prod_{i = 0}^{\ell - 2} \sigma^i(\omega^{\ell - i - 1})
\]
to $K(\chi_\alpha, \chi_\beta)$, where $\sigma$ is a generator of $\Gal(K(\sqrt[\ell]{\alpha})/K)$.
\end{lemma}

\begin{proof}
This is \cite[Theorem 3.1]{Michailov}.
\end{proof}

\begin{theorem}
\label{tHeisenberg}
There exists precisely one extension $M/\Q_\ell$ such that $\Gal(M/\Q_\ell)$ is isomorphic to $\textup{Heis}_\ell$. Its discriminant ideal equals
\[
(\ell)^{\ell(\ell + 1)(2\ell - 2)}.
\]
\end{theorem}

\begin{proof}
Since $\Q_\ell^\ast/\Q_\ell^{\ast \ell}$ is a $2$-dimensional vector space, it follows from local class field theory that there are two linearly independent characters $G_{\Q_\ell} \rightarrow \mathbb{F}_\ell$. In particular it follows that there is precisely one extension $K$ of $\Q_\ell$ with $\Gal(K/\Q_\ell) \cong \mathbb{F}_\ell^2$. We apply the inflation--restriction long exact sequence, see (\ref{eInfRes}), to deduce that
\[
\frac{\text{Hom}(G_K, \mathbb{F}_\ell)^{\Gal(K/\Q_\ell)}}{\text{Hom}(G_{\Q_\ell}, \mathbb{F}_\ell)} \cong \text{ker}(H^2(\Gal(K/\Q_\ell), \mathbb{F}_\ell) \xrightarrow{\text{inf}} H^2(G_{\Q_\ell}, \mathbb{F}_\ell)).
\]
The image of $\text{Hom}(G_{\Q_\ell}, \mathbb{F}_\ell)$ in $\text{Hom}(G_K, \mathbb{F}_\ell)$ is trivial, since $K$ is the maximal elementary abelian extension of $\Q_\ell$ with exponent $\ell$. Then we get an isomorphism
\[
\text{Hom}(G_K, \mathbb{F}_\ell)^{\Gal(K/\Q_\ell)} \cong \text{ker}(H^2(\Gal(K/\Q_\ell), \mathbb{F}_\ell) \xrightarrow{\text{inf}} H^2(G_{\Q_\ell}, \mathbb{F}_\ell)).
\]
But recall that the Heisenberg extensions form an $1$-dimensional subspace of
\[
H^2(\Gal(K/\Q_\ell), \mathbb{F}_\ell) \cong H^2(\mathbb{F}_\ell^2, \mathbb{F}_\ell).
\]
This shows that there is at most one such extension $M$. 

Let $\chi_{\text{un}}: G_{\Q_\ell} \rightarrow \mathbb{F}_\ell$ be a non-trivial unramified character and let $\chi_{\text{ram}}: G_{\Q_\ell} \rightarrow \mathbb{F}_\ell$ be a ramified character. The existence of $M$ is equivalent to the vanishing of $\theta_{\chi_{\text{un}}, \chi_{\text{ram}}}(\sigma, \tau)$ in $H^2(G_{\Q_\ell}, \mathbb{F}_\ell)$. There are natural maps
\[
H^2(G_{\Q_\ell}, \mathbb{F}_\ell) \xrightarrow{\text{res}} H^2(G_{\Q_\ell(\zeta_\ell)}, \mathbb{F}_\ell) \xrightarrow{\text{cores}} H^2(G_{\Q_\ell}, \mathbb{F}_\ell).
\]
The composition $\text{cores} \circ \text{res}$ is multiplication by $[\Q_\ell(\zeta_\ell) : \Q_\ell] = \ell - 1$. Hence the map $\text{res}$ is injective. Over $\Q_\ell(\zeta_\ell)$ we see that $\chi_{\text{ram}}$ is in the span of $\chi_{\text{un}}$ and the character $\chi_{\zeta_\ell}$ corresponding to the extension $\Q_\ell(\zeta_{\ell^2})/\Q_\ell(\zeta_\ell)$. By local class field theory we know that the norm map $\mathcal{O}_{\Q_\ell(\zeta_\ell)(\chi_{\text{un}})}^\ast \rightarrow \mathcal{O}_{\Q_\ell(\zeta_\ell)}^\ast$ is surjective. Since $\zeta_\ell$ is a unit, it follows from Lemma \ref{lMichailov} that $\theta_{\chi_{\text{un}}, \chi_{\text{ram}}}(\sigma, \tau)$ is trivial in $H^2(G_{\Q_\ell}, \mathbb{F}_\ell)$ as desired.

We now compute the discriminant of $M$. Define $L := \Q_\ell(\zeta_\ell)(\chi_{\text{un}})$. Take $\omega$ to be an element of $\mathcal{O}_L^\ast$ such that ${\rm N}_{L/\Q_\ell(\zeta_\ell)}(\omega) = \zeta_\ell$. Observe that $\zeta_\ell - 1$ is a uniformizer of $\Q_\ell(\zeta_\ell)$ and therefore also of $L$. Now we expand
\[
\omega = a_0 + a_1 (\zeta_\ell - 1) + a_2 (\zeta_\ell - 1)^2 + \cdots,
\]
where the digits $a_i$ are the Teichm\"uller lifts of $\mathbb{F}_{\ell^\ell}$ in $L$. Then 
\[
{\rm N}_{L/\Q_\ell(\zeta_\ell)}(\omega) = \zeta_\ell = 1 + (\zeta_\ell - 1)
\]
implies that
\[
a_0 \sigma(a_0) \cdot \ldots \cdot \sigma^{\ell - 1}(a_0) = 1
\]
with $\sigma$ a generator of $\Gal(L/\Q_\ell(\zeta_\ell))$. Now define $\omega_1 := \omega/a_0$, which still satisfies ${\rm N}_{L/\Q_\ell(\zeta_\ell)}(\omega_1) = \zeta_\ell$. Expand $\omega_1$ as
\[
\omega_1 = 1 + b_1 (\zeta_\ell - 1) + b_2 (\zeta_\ell - 1)^2 + \cdots.
\]
From ${\rm N}_{L/\Q_\ell(\zeta_\ell)}(\omega_1) = \zeta_\ell = 1 + (\zeta_\ell - 1)$ we deduce that
\begin{align}
\label{eHyperplane}
\sum_{i = 0}^{\ell - 1} \sigma^i(b_1) = 1.
\end{align}
Consider the element
\[
\prod_{i = 0}^{\ell - 2} \sigma^i(\omega_1^{\ell - i - 1}) = 1 + \left(\sum_{i = 0}^{\ell - 2} (\ell - i - 1) \sigma^i(b_1)\right)(\zeta_\ell - 1) + \cdots
\]
We claim that
\[
\sum_{i = 0}^{\ell - 2} (\ell - i - 1) \sigma^i(b_1)
\]
does not reduce to an element in $\mathbb{F}_\ell$ modulo the maximal ideal of $\mathcal{O}_L$. Suppose otherwise. Write $\text{red}_L$ for this natural reduction map. Take a normal basis $\eta, \sigma(\eta), \dots, \sigma^{\ell - 1}(\eta)$ of the field extension $\mathbb{F}_{\ell^\ell}/\mathbb{F}_\ell$, where we continue to write $\sigma$ for the natural induced automorphism of $\Gal(\mathbb{F}_{\ell^\ell}/\mathbb{F}_\ell)$ by $\sigma$. We now study the linear map $A: \mathbb{F}_{\ell^\ell} \rightarrow \mathbb{F}_{\ell^\ell}$ given by
\[
b \mapsto \sum_{i = 0}^{\ell - 2} (\ell - i - 1) \sigma^i(b).
\]
With respect to the basis $\eta, \sigma(\eta), \dots, \sigma^{\ell - 1}(\eta)$ our linear map $A$ becomes
\begin{align}
\label{eMatrixForm}
\begin{pmatrix}
\ell - 1 & 0 & \cdots & \ell - 2 \\
\ell - 2 & \ell - 1 & \cdots & \ell - 3 \\
\vdots  & \vdots  & \ddots & \vdots  \\
1 & 2 & \cdots & 0 \\
0 & 1 & \cdots & \ell - 1 
\end{pmatrix}
\end{align}
in matrix form. In view of equation (\ref{eHyperplane}), we get the desired contradiction if we are able to show that
\[
\{v \in \mathbb{F}_\ell^\ell : Av \in \langle (1, \dots, 1) \rangle\} \subseteq \{(x_1, \dots, x_n) : x_1 + \dots + x_n = 0\}.
\]
We have a decomposition
\[
\{v \in \mathbb{F}_\ell^\ell : Av \in \langle (1, \dots, 1) \rangle\} = \text{ker}(A) \oplus \langle (-1, 1, 0, \dots, 0) \rangle.
\]
Since $(-1, 1, 0, \dots, 0)$ is in the sum zero space, it suffices to prove that
\[
\text{ker}(A) \subseteq \{(x_1, \dots, x_n) : x_1 + \dots + x_n = 0\},
\]
which is equivalent to
\[
\text{im}(A^T) \supseteq \langle (1, \dots, 1) \rangle
\]
after taking orthogonal complements. This is indeed the case as we can see from multiplying the vector $(1,- 1, 0, \dots, 0)$ with the matrix in equation (\ref{eMatrixForm}). 

Having established the claim, we write
\[
\omega_2 := \prod_{i = 0}^{\ell - 2} \sigma^i(\omega_1^{\ell - i - 1}), \quad \omega_2 = 1 + c_1 (\zeta_\ell - 1) + c_2 (\zeta_\ell - 1)^2 + \cdots,
\]
where $\text{red}_L(c_1) \not \in \mathbb{F}_\ell$. From Lemma \ref{lMichailov} we see that 
\[
L(\zeta_{\ell^2}, \sqrt[\ell]{\omega_2})/\Q_\ell(\zeta_\ell)
\]
is a Heisenberg extension. But so is the extension $M\Q_\ell(\zeta_\ell)/\Q_\ell(\zeta_\ell)$. Then, by the long exact sequence (\ref{eInfRes}) it follows that 
\[
L(\zeta_{\ell^2}, \sqrt[\ell]{t \omega_2}) = M\Q_\ell(\zeta_\ell)
\]
for some twist $t \in \Q_\ell(\zeta_\ell)^\ast$.

Suppose that $t = (\zeta_\ell - 1)^s \cdot u$, where $u \in \mathcal{O}_{\Q_\ell(\zeta_\ell)}^\ast$ and $s \in \Z$. We claim that $\ell \mid s$. Assume for the sake of contradiction that $\ell \nmid s$. Denote by $\rho$ the automorphism of $L$ that sends $\zeta_{\ell}$ to $\zeta_{\ell}^2$ but fixes the field corresponding to $\chi_{\text{un}}$. We claim that $L(\zeta_{\ell^2})$, $L(\sqrt[\ell]{t \omega_2})$ and $L(\sqrt[\ell]{\rho(t \omega_2)})$ are three independent extensions in this case, which is impossible as $M\Q_\ell(\zeta_\ell)/L$ is bicyclic. Indeed, suppose that
\[
\zeta_\ell^{x_1} (t \omega_2)^{x_2} (\rho(t \omega_2))^{x_3} \in L^{\ast \ell}.
\]
Inspecting valuations, we certainly find that $x_2 + x_3 \equiv 0 \bmod \ell$. Then modulo $\ell$-th powers, the above becomes
\[
\zeta_\ell^{x_1} \omega_2^{x_2} \rho(\omega_2)^{x_3} u' \in L^{\ast \ell}
\]
with $u' \in \mathcal{O}_{\Q_\ell(\zeta_\ell)}^\ast$, which we expand as
\[
\zeta_\ell^{x_1} u' \cdot \left(1 + c_1(x_2 + x_3(\zeta_\ell + 1)) (\zeta_\ell - 1) + \cdots\right).
\]
Since $\text{red}_L(c_1) \not \in \mathbb{F}_\ell$, it follows that 
\[
\text{red}_L(c_1(x_2 + x_3(\zeta_\ell + 1))) \not \in \mathbb{F}_\ell \quad \text{ or } \quad \text{red}_L(x_2 + x_3(\zeta_\ell + 1)) = 0.
\]
We first dispose with the second case. But $\text{red}_L(x_2 + x_3(\zeta_\ell + 1)) = 0$ implies that $x_2 + 2x_3 \equiv 0 \bmod \ell$ and hence $x_2 \equiv x_3 \equiv 0 \bmod \ell$. In this case we conclude that 
\[
x_1 \equiv x_2 \equiv x_3 \equiv 0 \bmod \ell
\]
as desired. From now on we suppose that
\[
\text{red}_L(c_1(x_2 + x_3(\zeta_\ell + 1))) \not \in \mathbb{F}_\ell.
\]
In this case $\zeta_\ell^{x_1} \omega_2^{x_2} \rho(\omega_2)^{x_3} u'$ is of the shape
\begin{align}
\label{eU1}
d_0 + d_1 (\zeta_\ell - 1) + \cdots \quad \text{ with } \text{red}_L(d_0) \in \mathbb{F}_\ell \setminus \{0\} \text{ and } \text{red}_L(d_1) \not \in \mathbb{F}_\ell,
\end{align}
where the digits $d_i$ are the Teichm\"uller lifts. We claim that such elements are never $\ell$-th powers in $L$. Suppose that $\alpha$ is such an element and consider the polynomial
\[
f(x) = x^\ell - \alpha.
\]
Then $f(x + d_0)$ is irreducible by Eisenstein's criterion. This finishes the proof of both claims, and we conclude that $\ell \mid s$. Furthermore, $\omega_3 := \frac{t \omega_2}{(1 - \zeta_\ell)^s}$ has an expansion of the shape displayed in equation (\ref{eU1}).

Finally, we compute the discriminant of the extension $L(\sqrt[\ell]{\omega_3})/L$. We just showed that
\[
f(x + d_0) = (x + d_0)^\ell - \omega_3 = -\omega_3 + \sum_{i = 0}^\ell \binom{\ell}{i} x^i d_0^{\ell - i}
\]
is Eisenstein, i.e. $f(x + d_0)$ satisfies Eisenstein's criterion. Write $r$ for a root of the polynomial $f(x + d_0)$. Since $f(x + d_0)$ is Eisenstein, it follows that
\[
\mathcal{O}_{L(\sqrt[\ell]{\omega_3})} = \mathcal{O}_L[r],
\]
so we are in the position to apply Lemma \ref{lDifferent} part (v). We conclude that
\[
\mathfrak{d}_{L(\sqrt[\ell]{\omega_3})/L} = \left(\sum_{i = 1}^{\ell} \binom{\ell}{i} i r^{i - 1} d_0^{\ell - i}\right) = (\ell).
\]
By construction we have that $L(\zeta_{\ell^2}, \sqrt[\ell]{\omega_3}) = M\Q_\ell(\zeta_\ell)$. Then there exists some degree $\ell$ cyclic extension $M'$ of $\Q_\ell(\chi_{\text{un}})$ such that the Galois closure of $M'$ is $M$ and furthermore $M' \subseteq L(\sqrt[\ell]{\omega_3})$. This implies that
\[
\Delta_{M'/\Q_\ell(\chi_{\text{un}})}^{\ell - 1} {\rm N}_{M'/\Q_\ell(\chi_{\text{un}})}(\Delta_{L(\sqrt[\ell]{\omega_3})/M'}) = \Delta_{L/\Q_\ell(\chi_{\text{un}})}^\ell {\rm N}_{L/\Q_\ell(\chi_{\text{un}})}(\Delta_{L(\sqrt[\ell]{\omega_3})/L})
\]
The extensions $L(\sqrt[\ell]{\omega_3})/M'$ and $L/\Q_\ell(\chi_{\text{un}})$ are tamely ramified and of degree $\ell - 1$. We conclude that
\[
\Delta_{M'/\Q_\ell(\chi_{\text{un}})}^{\ell - 1} \cdot (\ell)^{\ell - 2} = (\ell)^{(\ell - 2) \ell} \cdot (\ell)^{(\ell - 1) \ell}
\]
and hence
\begin{align}
\label{eLBad}
\Delta_{M'/\Q_\ell(\chi_{\text{un}})} = (\ell)^{2\ell - 2}.
\end{align}
Lemma \ref{lBicyclic} yields
\[
\Delta_{M/\Q_\ell(\chi_{\text{un}})} = \prod_{i = 1}^{\ell + 1} \Delta_{M_i/\Q_\ell(\chi_{\text{un}})},
\]
where the $M_i$ are the subfields $\Q_\ell(\chi_{\text{un}}) \subsetneq M_i \subsetneq M$ of the bicyclic extension $M/\Q_\ell(\chi_{\text{un}})$. One of the $M_i$ is the field $\Q_\ell(\chi_{\text{un}}, \chi_{\text{ram}})$, while the other $M_i$ are all isomorphic to $M'$ by Lemma \ref{lHIsoTest}. We deduce that
\[
\Delta_{M/\Q_\ell(\chi_{\text{un}})} = (\ell)^{\ell(2\ell - 2)} \cdot \Delta_{\Q_\ell(\chi_{\text{un}}, \chi_{\text{ram}})/\Q_\ell(\chi_{\text{un}})} = (\ell)^{(\ell + 1)(2\ell - 2)}
\]
as desired.
\end{proof}

\subsection{Minimal Heisenberg extensions}
In this subsection we will study Heisenberg extensions from a global perspective. We start by defining minimal Heisenberg extensions, which is analogous to the definition of minimal dihedral extensions given by Stevenhagen \cite{Stevenhagen}.

\begin{definition}
Let $\chi, \chi': G_\Q \rightarrow \mathbb{F}_\ell$ be two linearly independent characters. Let $M$ be a Heisenberg extension of $\Q$ containing $\Q(\chi, \chi')$. We say that $M$ is minimal if the following two conditions are satisfied
\begin{itemize}
\item $M$ is unramified at every place $v$ that is unramified in $\Q(\chi, \chi')$;
\item $M/\Q(\chi, \chi')$ is unramified at all primes above $\ell$ in case $\ell$ has residue field degree $1$ in $\Q(\chi, \chi')$.
\end{itemize}
\end{definition}

Suppose that the residue field degree of $\ell$ in $\Q(\chi, \chi')$ is $1$ and further assume that $\ell$ ramifies in $\Q(\chi, \chi')$. As we shall see, the second condition is then automatically satisfied for all Heisenberg extensions $M$ containing $\Q(\chi, \chi')$. From this it follows that any Heisenberg extension $M$ that satisfies the first condition also satisfies the second condition.

\begin{lemma}
\label{lHeisenbergExists}
Let $\chi, \chi': G_\Q \rightarrow \mathbb{F}_\ell$ be two linearly independent characters. Then $\theta_{\chi, \chi'}(\sigma, \tau)$ is trivial in $H^2(G_\Q, \mathbb{F}_\ell)$ if and only if all ramified primes not equal to $\ell$ have residue field degree $1$ in $\Q(\chi, \chi')$.
\end{lemma}

\begin{proof}
We first prove the backward implication. There are natural maps
\[
H^2(G_\Q, \mathbb{F}_\ell) \xrightarrow{\text{res}} H^2(G_{\Q(\zeta_\ell)}, \mathbb{F}_\ell) \xrightarrow{\text{cores}} H^2(G_\Q, \mathbb{F}_\ell).
\]
The composition $\text{cores} \circ \text{res}$ is multiplication by $[\Q(\zeta_\ell) : \Q] = \ell - 1$. It follows that the map $\text{res}$ is injective. From class field theory, we get another injective map
\[
H^2(G_{\Q(\zeta_\ell)}, \mathbb{F}_\ell) \rightarrow \bigoplus_w H^2(G_{\Q(\zeta_\ell)_w}, \mathbb{F}_\ell),
\]
where $w$ runs over the places of $\Q(\zeta_\ell)$. Hence it suffices to check that $\theta_{\chi, \chi'}(\sigma, \tau)$ is trivial in $H^2(G_{\Q(\zeta_\ell)_w}, \mathbb{F}_\ell)$ for each place $w$.

Denote by $v$ the place of $\Q$ below $w$. If $v$ is unramified in $\Q(\chi, \chi')$ or if $v$ is the infinite place, it is clear that $\theta_{\chi, \chi'}(\sigma, \tau)$ is trivial in $H^2(G_{\Q(\zeta_\ell)_w}, \mathbb{F}_\ell)$. Now suppose that $v \neq \ell$ ramifies in $\Q(\chi, \chi')$. 

By assumption $v$ has residue field degree $1$ in $\Q(\chi ,\chi')$. If $\chi$ and $\chi'$ are both ramified at $v$, then the $2$-cocycle $\theta_{\chi, \chi'}(\sigma, \tau)$ is trivial in $H^2(G_{\Q(\zeta_\ell)_w}, \mathbb{F}_\ell)$ by Remark \ref{rSelfcup}, since it is of the shape $\theta_{\rho, \rho}(\sigma, \tau)$ locally at $v$. If instead $\chi$ is ramified at $v$, while $\chi'$ is not, the $2$-cocycle $\theta_{\chi, \chi'}(\sigma, \tau)$ is the zero map locally at $v$. 

It remains to deal with the case $v = \ell$. But the analysis in Theorem \ref{tHeisenberg} shows that $\theta_{\chi, \chi'}(\sigma, \tau)$ is always locally trivial at $\ell$. For the forward implication, we reverse the above logic. This completes the proof.
\end{proof}

Let $\chi, \chi': G_\Q \rightarrow \mathbb{F}_\ell$ be two linearly independent characters. We define
\[
\mu(\chi, \chi') = 
\left\{
\begin{array}{ll}
\ell^0 & \mbox{if } \Q(\chi, \chi') \text{ is unramified at } \ell \\
\ell^{(\ell - 1)(2\ell - 2)} & \mbox{if } \ell \text{ splits in } \Q(\chi) \text{ and ramifies in } \Q(\chi, \chi') \\
\ell^{\ell(2\ell - 2)} & \mbox{if } \ell \text{ is inert in } \Q(\chi) \text{ and ramifies in } \Q(\chi, \chi') \\
\ell^{\ell(2\ell - 2)} & \mbox{if } \ell \text{ ramifies in } \Q(\chi) \text{ and splits in } \Q(\chi, \chi') \\
\ell^{(\ell + 1)(2\ell - 2)} & \mbox{if } \ell \text{ ramifies in } \Q(\chi) \text{ and is inert in } \Q(\chi, \chi').
\end{array}
\right.
\]
Denote by $\widetilde{\Delta}(\chi)$ the product of the ramifying primes in $\Q(\chi)$ that are coprime to $\ell$ and let $\text{free}(d, a)$ be the largest squarefree integer dividing $d$ and coprime with $a$.

\begin{theorem}
\label{tMinimalDisc}
Let $\chi, \chi': G_\Q \rightarrow \mathbb{F}_\ell$ be two linearly independent characters. Suppose that $\theta_{\chi, \chi'}(\sigma, \tau)$ is trivial in $H^2(G_\Q, \mathbb{F}_\ell)$. Then there exists a minimal Heisenberg extension $M/\Q$ containing $\Q(\chi, \chi')$, which equals $\Q(\chi, \chi')(\rho)$ for some $\rho \in \textup{Heis}(\Q(\chi, \chi')/\Q)$. Furthermore, all Heisenberg extensions containing $\Q(\chi, \chi')$ are obtained by twisting $\rho$ by a character $\chi'': G_\Q \rightarrow \mathbb{F}_\ell$.

Now suppose that $\Q(\chi) \subsetneq L \subsetneq M$ and suppose that the Galois closure of $L$ is $M$.  Then
\[
\Delta_{L/\Q} = \widetilde{\Delta}(\chi)^{\ell (\ell - 1)} \textup{free}(\widetilde{\Delta}(\chi'), \widetilde{\Delta}(\chi))^{(\ell - 1)^2} \mu(\chi, \chi').
\]
\end{theorem}

\begin{proof}
By Lemma \ref{lHeisenbergExists} it follows that there exists a Heisenberg extension $M$ of $\Q$ containing $\Q(\chi, \chi')$. It is then a general fact about central extensions that there exists a Heisenberg extension $M$ containing $\Q(\chi, \chi')$ that is unramified at every place $v$ that is unramified in $\Q(\chi, \chi')$, see \cite[Proposition 4.8]{Ko--Pa}. We claim that such an extension $M$ is minimal.

It remains to analyze the splitting behavior of $v = \ell$, where $v$ has residue field degree $1$ in $K := \Q(\chi, \chi')$. By the previous remark we may and will assume that $v$ ramifies in $K$. Let $w$ be a place of $\Q(\chi, \chi')$ above $v$. We are going to show that $M$ is unramified at $w$. Consider the commutative diagram
\[ 
\begin{tikzcd}[row sep=large, column sep = 1.75em]
\text{Hom}(G_K, \mathbb{F}_\ell)^{\Gal(K/\Q)} \arrow[swap]{d}{\text{res}} \arrow{rr}{\text{tr}} && H^2(\Gal(K/\Q), \mathbb{F}_\ell) \arrow{d}{\text{res}} \\%
\text{Hom}(G_{K_w}, \mathbb{F}_\ell)^{\Gal(K_w/\Q_v)} \arrow{r}{\text{tr}} & H^2(\Gal(K_w/\Q_v), \mathbb{F}_\ell) \arrow{r}{\cong} & H^2(\Gal(K/D_{w/v}), \mathbb{F}_\ell),
\end{tikzcd}
\]
where $D_{w/v}$ is the decomposition group. To check that the diagram is commutative, we remark that equation (\ref{eExt}) shows that the transgression map is explicitly given by sending the character $\rho \in \text{Hom}(G_K, \mathbb{F}_\ell)^{\Gal(K/\Q)}$ to the $2$-cocycle
\[
(\sigma_1, \sigma_2) \mapsto \rho\left(\widetilde{\sigma_1 \sigma_2}, \widetilde{\sigma_1}^{-1} \widetilde{\sigma_2}^{-1}\right),
\]
where we fix lifts $\widetilde{\sigma} \in G_\Q$ for every $\sigma \in \Gal(K/\Q)$. As we have seen in the proof of Lemma \ref{lHeisenbergExists}, the class of $\theta_{\chi, \chi'}(\sigma, \tau)$ is trivial in $H^2(\Gal(K_w/\Q_v), \mathbb{F}_\ell)$. Writing $\rho$ for a character in $\text{Hom}(G_K, \mathbb{F}_\ell)^{\Gal(K/\Q)}$ defining $M$, it follows that $\rho$ becomes a character from $\Q_v$ when restricted to $K_w$. This implies the claim, since $v$ ramifies in $K$.

Having established the claim, we have shown the existence of a minimal Heisenberg extension $M$ containing $K$. From the inflation--restriction sequence (\ref{eInfRes}) it is immediate that any other Heisenberg extension of $K$ is obtained by twisting $\rho$ by a character $\chi'': G_\Q \rightarrow \mathbb{F}_\ell$.

We will now further analyze the ramification properties of $M$. Take a place $v \neq \ell$ that ramifies in $K$. Let $w$ be a place of $K$ above $v$. We claim that $w$ is unramified in $M$. If not, we see that any inertia subgroup $I_v$ of $v$ must be of size $\ell^2$. But $v$ is tamely ramified and therefore $I_v$ is a cyclic group. This is plainly impossible, since every element of the Heisenberg group has order $\ell$.

We are now ready to compute the discriminant of $L$. Take a place $v \neq \ell$ that ramifies in $\Q(\chi)$ and recall the formula
\[
\Delta_{L/\Q} = {\rm N}_{\Q(\chi)/\Q}(\Delta_{L/\Q(\chi)}) \Delta_{\Q(\chi)/\Q}^\ell.
\]
From the above we see that the $v$-adic valuation of ${\rm N}_{\Q(\chi)/\Q}(\Delta_{L/\Q(\chi)})$ is $0$. Furthermore, since $v$ is tamely ramified, we have that the $v$-adic valuation of $\Delta_{\Q(\chi)/\Q}^\ell$ is $\ell(\ell - 1)$. Next we compute the contribution from those $v \neq \ell$ that are unramified in $\Q(\chi)$ but ramify in $K$. In this case the formula simplifies to
\[
\Delta_{L/\Q} = {\rm N}_{\Q(\chi)/\Q}(\Delta_{L/\Q(\chi)}).
\]
Furthermore, we know by Lemma \ref{lHeisenbergExists} that $v$ splits completely in $\Q(\chi)$. Suppose that $w_1, \dots, w_\ell$ are the places above $v$. Because $w_1, \dots, w_\ell$ ramify in $K$ but do not ramify further in $M/K$, it follows that precisely $\ell - 1$ of them must ramify in $L$ so that the $v$-adic valuation of $\Delta_{L/\Q}$ is $(\ell - 1)^2$.

It remains to deal with the case $v = \ell$. We distinguish four cases
\begin{enumerate}
\item[(i)] suppose that $\ell$ ramifies in $\Q(\chi)$ and has residue field degree $1$ in $K$. In this case any prime above $\ell$ is unramified in $L$. Hence
\[
v_\ell(\Delta_{L/\Q}) = \ell(2\ell - 2);
\]
\item[(ii)] suppose that $\ell$ splits in $\Q(\chi)$ but ramifies in $K$. Then
\[
v_\ell(\Delta_{L/\Q}) = {\rm N}_{\Q(\chi)/\Q}(\Delta_{L/\Q(\chi)}).
\]
Note that
\[
v_\ell(\Delta_{M/\Q}) = \ell^2(2\ell - 2)
\]
and hence $w(\Delta_{M/\Q(\chi)}) = \ell(2\ell - 2)$ for any place $w$ of $\Q(\chi)$ above $v$. Suppose that $w$ ramifies in $L$. Consider the bicyclic extension $M/\Q(\chi)$. There are $\ell + 1$ intermediate fields $K, L_1, \dots, L_\ell$, where the $L_i$ are all isomorphic by Lemma \ref{lHIsoTest}. Furthermore, $w$ ramifies in $K$ and precisely $\ell - 1$ of the $L_i$. Therefore Lemma \ref{lBicyclic} implies that
\[
(\ell - 1) \cdot w(\Delta_{L/\Q(\chi)}) + w(\Delta_{K/\Q(\chi)}) = w(\Delta_{M/\Q(\chi)}) = \ell(2\ell - 2).
\]
We conclude that
\[
w(\Delta_{L/\Q}) = 2\ell - 2, \quad v_\ell(\Delta_{L/\Q}) = (\ell - 1)(2\ell - 2);
\]
\item[(iii)] suppose that $\ell$ ramifies in $\Q(\chi)$ and has residue field degree $\ell$ in $K$. Denote by $w$ the unique place of $\Q(\chi)$ above $\ell$. Arguing as above we get
\[
\ell \cdot w(\Delta_{L/\Q(\chi)}) = w(\Delta_{M/\Q(\chi)}) = \ell(2\ell - 2),
\]
where the last equality follows from Theorem \ref{tHeisenberg}. Hence we have
\[
v_\ell(\Delta_{L/\Q}) = (2\ell - 2) + \ell(2\ell - 2) = (\ell + 1)(2\ell - 2);
\]
\item[(iv)] suppose that $\ell$ is inert in $\Q(\chi)$ but ramifies in $K$. Inspecting the proof of Theorem \ref{tHeisenberg}, see equation (\ref{eLBad}), we conclude that
\[
v_\ell(\Delta_{L/\Q}) = \ell(2\ell - 2).
\]
\end{enumerate}
This completes the proof.
\end{proof}

\subsection{Counting Heisenberg extensions by discriminant}
Let $\chi, \chi': G_\Q \rightarrow \mathbb{F}_\ell$ be two linearly independent characters. Define for an integer $d > 0$
\[
\mu(\chi, \chi', d) = 
\left\{
\begin{array}{ll}
\ell^{ {\ell} (2\ell - 2)} & \mbox{if } \Q(\chi, \chi') \text{ is unramified at } \ell \text{ and } \ell \mid d\\
\mu(\chi, \chi') & \mbox{otherwise.} \\
\end{array}
\right.
\]
We also put
\begin{align*}
D(d, \chi, \chi', \ell) &:= \widetilde{\Delta}(\chi)^{\ell (\ell - 1)} \textup{free}(\widetilde{\Delta}(\chi'), \widetilde{\Delta}(\chi))^{(\ell - 1)^2} \mu(\chi, \chi', d) \\
S_1(X, \ell) &:= \{d \in \Z_{>0} : d \leq X, \ d \text{ squarefree}, \ p \mid d \Rightarrow p \equiv 0, 1 \bmod \ell\} \\
S_2(X, \chi, \chi', \ell) &:= \{d \in S_1(X, \ell) : \gcd(d, \widetilde{\Delta}(\chi) \widetilde{\Delta}(\chi')) = 1\} \\
S_3(X, \chi, \chi', \ell) &:= \sum_{\substack{d \in S_2(X, \chi, \chi', \ell) \\ \text{free}(d, \ell)^{\ell(\ell - 1)} \leq \frac{X}{D(d, \chi, \chi', \ell)}}} (\ell - 1)^{\omega^\ast_\ell(d)},
\end{align*}
where $\omega^\ast_\ell$ is the number of prime divisors (counted without multiplicity) not equal to $\ell$. Recall that $N(\text{Heis}_\ell, X)$ denotes the number of degree $\ell^2$ extensions $L$ of $\Q$, up to isomorphism, with $\Gal(N(L)/\Q) \cong \text{Heis}_\ell$ and absolute discriminant bounded by $X$.

\begin{theorem}
\label{tHDisc}
Let $\ell$ be an odd prime number. Then
\begin{align}
\label{eAlgHeiSum}
N(\textup{Heis}_\ell, X) = (\ell - 1)^{-2} \hspace{-0.5cm} \sum_{\substack{\chi, \chi': G_\Q \rightarrow \mathbb{F}_\ell \\ \chi, \chi' \textup{ lin. indep.}}} \mathbbm{1}_{\theta_{\chi, \chi'}(\sigma, \tau) \textup{ trivial}} \cdot \ell^{\omega(\widetilde{\Delta}(\chi) \widetilde{\Delta}(\chi')) - 3} \cdot S_3(X, \chi, \chi', \ell).
\end{align}
\end{theorem}

\begin{proof}
We recall that $\theta_{\chi, \chi'}(\sigma, \tau)$ and $\theta_{\chi, \chi' + a\chi}(\sigma, \tau)$ give the same class in $H^2(\mathbb{F}_\ell^2, \mathbb{F}_\ell)$ for all $a \in \mathbb{F}_\ell$ by Remark \ref{rSelfcup}.

First, we fix $\chi$ and compute the contribution from those degree $\ell^2$ Heisenberg extensions $L$ containing $\Q(\chi)$. Since $\chi$ and $a \chi$ both have fixed field $\Q(\chi)$ for any $a \in \mathbb{F}_\ell^\ast$, we are overcounting by a factor $\ell - 1$. Next, let us further restrict to those $L$ such that the normal closure of $L$ contains $\Q(\chi, \chi')$ with $\chi'$ linearly independent from $\chi$. This certainly implies that $\theta_{\chi, \chi'}(\sigma, \tau)$ is trivial. 

Hence further fix a $\chi'$ linearly independent from $\chi$ with $\theta_{\chi, \chi'}(\sigma, \tau)$ trivial. Note that there are in fact $\ell(\ell - 1)$ choices of $\chi'$ that all give the same bicyclic extension $\Q(\chi, \chi')$, namely $a\chi' + b\chi$ with $a \in \mathbb{F}_\ell^\ast$ and $b \in \mathbb{F}_\ell$. Hence we are overcounting by another factor $\ell(\ell - 1)$.

Now we compute the contribution from the fields $L$ containing $\Q(\chi)$ such that the normal closure of $L$ contains $\Q(\chi, \chi')$. Fix a minimal extension $M$ containing $\Q(\chi, \chi')$. Then any field $L'$ satisfying $\Q(\chi) \subsetneq L' \subsetneq M$ has discriminant
\[
\widetilde{\Delta}(\chi)^{\ell (\ell - 1)} \textup{free}(\widetilde{\Delta}(\chi'), \widetilde{\Delta}(\chi))^{(\ell - 1)^2} \mu(\chi, \chi')
\]
by Theorem \ref{tMinimalDisc}. Let $\rho \in \text{Heis}(\Q(\chi, \chi')/\Q)$ be a character with fixed field $M$. Twisting $\rho$ by characters $\chi'': G_\Q \rightarrow \mathbb{F}_\ell$, we get all degree $\ell^3$ Heisenberg extensions containing $\Q(\chi, \chi')$. However, we get every extension $\ell^2$ times, since the characters $\chi$ and $\chi'$ are trivial when restricted to $G_{\Q(\chi, \chi')}$.

Suppose that we twist $\rho$ by a character $\chi'': G_\Q \rightarrow \mathbb{F}_\ell$ that is ramified precisely at the primes dividing $d$. From class field theory we immediately get that $d \in S_1(\infty, \ell)$. Furthermore for such an integer $d$, there are precisely $(\ell - 1)^{\omega(d)}$ characters that are ramified at exactly those primes dividing $d$. We claim that the discriminant of any field $L'$ such that $\Q(\chi) \subsetneq L' \subsetneq \Q(\chi, \chi')(\rho + \chi'')$ equals
\[
\widetilde{\Delta}(\chi)^{\ell (\ell - 1)} \textup{free}(\widetilde{\Delta}(\chi'), \widetilde{\Delta}(\chi))^{(\ell - 1)^2} \text{free}(d, \ell \widetilde{\Delta}(\chi) \widetilde{\Delta}(\chi'))^{\ell (\ell - 1)} \mu(\chi, \chi', d).
\]
The factor $\text{free}(d, \ell \widetilde{\Delta}(\chi) \widetilde{\Delta}(\chi'))^{\ell(\ell - 1)}$ is easily computed. Let us now focus on the factor $\mu(\chi, \chi', d)$. If there is precisely one place above $\ell$ in $\Q(\chi, \chi')$, twisting does not change the discriminant locally at $\ell$ by Theorem \ref{tHeisenberg}. Indeed, the two twists have the same normal closure (since there is only one Heisenberg field locally at $\ell$) and share the same cyclic subfield, so we can apply Lemma \ref{lHIsoTest}. Similarly, if $\ell$ ramifies in $\Q(\chi)$, twisting does not change the discriminant locally at $\ell$. If $\ell$ splits in $\Q(\chi)$ and ramifies in $\Q(\chi, \chi')$, then
\[
L \otimes \Q_\ell \cong \Q_\ell(\chi') \oplus \dots \oplus \Q_\ell(\chi') \oplus \Q_\ell^{\ell}
\]
or
\[
L \otimes \Q_\ell \cong \Q_\ell(\chi_{\text{un}}) \oplus \Q_\ell(\chi_{\text{un}} + \chi') \oplus \dots \oplus \Q_\ell(\chi_{\text{un}} + (\ell - 1)\chi'),
\]
where $\chi_{\text{un}}$ is an unramified degree $\ell$ character of $G_{\Q_\ell}$. Since $\chi'$ is a ramified character, we see once more that twisting does not change the discriminant locally at $\ell$. A similar analysis works if $\ell$ is unramified in $\Q(\chi, \chi')$.

Having established the claim, we are now ready to complete the proof. There are
\[
\ell^{\omega(\widetilde{\Delta}(\chi) \widetilde{\Delta}(\chi'))} 
\]
characters only ramified at the places dividing $\widetilde{\Delta}(\chi) \widetilde{\Delta}(\chi')$. Twisting with such characters clearly does not change the discriminant. Furthermore, they give
\[
\ell^{\omega(\widetilde{\Delta}(\chi) \widetilde{\Delta}(\chi')) - 2} 
\]
different fields, because the characters $\chi$ and $\chi'$ are trivial characters of $G_{\Q(\chi, \chi')}$. This gives the theorem.
\end{proof}

\section{Analytic prerequisites}   
\subsection{The general question} 
\label{Thegeneral}
From now on we shall mostly focus on the case $\ell = 3$. The aim of this section is to transform equation (\ref{eAlgHeiSum}) in the character sum ${\rm Heis} (X, 3)$ (see Proposition \ref{pAnaHeisSum} below). The definition of ${\rm Heis} (X, 3)$ is given in Definition \ref{definitionHeis} below. Since this character sum is rather delicate, we take some time to present its definition.

By convention  we reserve the letters $p$ and  $\ell$ for usual rational primes. The letter $r$ will also designate a prime particularly in Definition \ref{definition1} and in the formulas deduced from it. When $\ell \geq 3$ is a prime, we introduce the following sets of integers 
$$
 \PP_{\ell}:= \{ p : p\equiv  0,\, 1 \bmod \ell\}, $$
$$
\PP^*_{\ell} := \{ p : p\equiv 1 \bmod \ell\},
$$
$$
\N_{\ell}  :=\{ n : n\geq 1, \, n \text{ squarefree},  \, p\mid n \Rightarrow p \in \PP_{ \ell} \}, 
$$
and 
$$
\N_{\ell} ^*:=\{ n : n\geq 1, \,   n \text{ squarefree},\, p\mid n \Rightarrow p \in \PP^*_{ \ell} \}. 
$$
For $d\geq 1$, we denote by $\omega_{\ell}^* (d)$ the number of distinct prime divisors of $d$ belonging to $\PP_{\ell}^*$ and, as usual, $\omega (d) $ is the total number of distinct prime divisors of $d$.

\subsection{Standard primes, standard decomposition and characters} 
Let 
$$
j = \frac {-1 + i \sqrt 3}{2},
$$
be a cubic root of unity. For $z \in \Z [j]$, let ${\rm N}(z) =z\cdot \overline z$ be the norm of $z$. Every $p\in \mathbb P_{3}^*$  can be uniquely written as  
\begin{equation}
\label{standarddecomposition}
p=\pi \, \overline \pi 
\end{equation}
where 
\begin{equation*}
\begin{cases}
\pi \text{ and } \overline \pi \text{ belong to } \Z [j],\\
\pi \text{ is } primary \text{ (which means  } \pi \equiv 2 \bmod 3),  \\
{\rm Im } \,  \pi  >0.  
\end{cases}
\end{equation*}
This decomposition is named  the  {\it standard decomposition of $p$,}  and $\pi$ is a  {\it standard prime}.
For $p \in \mathbb P_{3}^*$, there are two Dirichlet characters modulo $p$ with order $3$. One of these
is 
\begin{equation}
\label{conventiononcharacter}
\chi_p (n) := \Bigl( \frac n\pi\Bigr)_3,
\end{equation}
which is defined without ambiguity as soon as $\pi$ is given by the standard decomposition \eqref{standarddecomposition}. Recall that the cubic character $\bigl( \frac \alpha  \pi\bigr)_3$ is defined, for $\alpha \in \Z [j]$ not divisible by $\pi$, by the formula 
$$
\Bigl( \frac \alpha \pi\Bigr)_3:= j^m,
$$
where $0\leq m\leq 2$ is the unique integer such that $\alpha^{\frac {p-1}3} \equiv j^m \bmod \pi$ (see \cite[Chap.9\,\S 3]{Ir-Ro}, for instance).

Modulo $9$, there are also two Dirichlet characters with order $3$. One of these is the character $\chi_3$ defined by its value
$$
\chi_3 (2) = j,
$$
which also defines $\chi_3$ without ambiguity. In conclusion, for every $p\in \PP_{3}$ we have fixed a Dirichlet character $\chi_p$ of order $3$. 

Let $f: \mathbb P_{3 }\longrightarrow \mathbb F_3$ be a function. By definition, the {\it support of $f$} is the set
$$
{\rm supp}\, f := \{ p \in \PP_3 : f(p) \not= 0\},
$$  
and ${\rm supp}_3 \, f$  is the support of the restriction of $f$ to $\PP_{3}^*$. We introduce the sets  of functions  
\begin{equation*}
V: = \{ f : \PP_3\longrightarrow \F_3, {\rm supp}\, f \text{ is finite} \}, 
\end{equation*}  
and
\begin{equation}\label{defV*}
V^* := \{ f : \PP_3\longrightarrow \F_3, {\rm supp}\, f \text{ is finite and } f(3)=0\}.
\end{equation}  
The sets $V$ and $V^*$ naturally have a structure of  $\F_3$--vector space with infinite dimension.

Given an $f$ in $V$, we define the Dirichlet character $\chi (f) $ over $\Z$ by the formula
\begin{equation}
\label{defchi(f)}
\chi (f) := \prod_{p\in \mathbb P_{3}} \chi_p ^{f(p)}.
\end{equation}
This has a meaning since this is a finite product and since all $\chi_p$ have order $3$. To evaluate $\chi (f)$ at some number $m \in \Z$, we naturally have
\begin{equation}
\label{convention}
\chi (f) (m) =\prod_{p \in \mathbb P_{3} }\bigl[\, \chi_p (m)\,\bigr]^{f(p)}
\end{equation}
with the convention that $z^0 =1$ for any $z \in \C$. In particular, we have
\begin{equation}
\label{chi(f)=0or}
\chi(f)(p)=
\begin{cases}
0 &\text{ if } p\in {\rm supp}\, f,\\
1,\, j\text{ or } j^2&\text{ if } p\notin {\rm supp}\, f.
\end{cases}
\end{equation}
To any $f\in  V$ we associate an integer $\Delta (f) \in \N_{3}^* $ defined by
$$
\Delta (f) := \prod_{ p \in \,{\rm supp}_3 f}  \   p.
$$
If $f(3) = 0$, then $\Delta (f)$ is the conductor of the Dirichlet character $\chi (f)$. On the other hand, if $f(3) \not= 0$, the conductor of $\chi (f)$ is equal to $9 \cdot \Delta (f)$. For $\Delta \in \N_3^*$, we will meet the following sets of functions, with cardinalities $3\cdot 2^{\omega (\Delta ) }$ and $2^{\omega (\Delta)}$
\begin{equation}
\label{defVV*}
V(\Delta) := \{ f \in V : \Delta (f) = \Delta\} \text{ and } V^* (\Delta):= \{ f \in V^*: \Delta (f) =\Delta\}.
\end{equation}
Finally, we introduce the function ${\mathbbm 1} (f, f')$ which can be interpreted as a characteristic function  since it takes  only values $0$ and  $1$ (see Lemma \ref{1leq1} below).   
 
\begin{definition}
\label{definition1}   
For $f, f'\in V$ let ${\mathbbm 1}(f, f')$ be the number defined by
$$
{\mathbbm 1} (f, f') := 3^{-\vert \, {\rm supp}_3\, f \, \cup\, {\rm supp}_3 \, f'\vert} \ \prod_{r \mid \Delta (f) \, \Delta (f')}
\Bigl(\,  \sum_{(z,z')\in \F_3^2\atop  f(r) z+f'(r) z'=0} \,
\bigl(\chi (zf+z'f')\bigr) (r)
\Bigr)
$$
\end{definition}

\noindent It follows from Lemma \ref{lHeisenbergExists} that 
\begin{align}
\label{eExpandIndicator}
{\mathbbm 1} (f, f') = \mathbbm{1}_{\theta_{\chi, \chi'}(\sigma, \tau) \text{ trivial}}.
\end{align}
In particular the following lemma is now obvious.

\begin{lemma} 
\label{1leq1}
For every $f$ and $f'$ in $V$, one has the equality 
$$
{\mathbbm 1}(f, f') \in \{0, 1\}.
$$
\end{lemma}

\subsection{\texorpdfstring{The $\mu$--functions}{The mu--functions}} 
To each pair $(f, f') \in V^2$ we associate an integer denoted by $\mu (f, f')$. This integer is a power of $3$ but it is not a symmetric function of $f$ and $f'$.

\begin{definition} 
For every $f$ and $f'$ in $V$, we define
$$
\mu (f, f')=
\begin{cases}
1 & \textup{ if } f(3)=f'(3) =0,\\
3^8 &\textup{ if } f(3)= 0, \, f' (3) \not= 0, \,  \textup{ and }  \chi (f) (3) =1,\\
3^{12}  &\textup{ if } f(3)= 0, \, f' (3) \not= 0, \,  \textup{ and }  \chi (f) (3) \in \{j, j^2\},\\
3^{12} &\textup{ if } f(3)\not= 0, \, f' (3) = 0, \,  \textup{ and }  \chi (f') (3) =1,\\
3^{16}  &\textup{ if } f(3)\not= 0, \, f' (3)  = 0, \,  \textup{ and }  \chi (f') (3) \in \{j, j^2\},\\
3^{12} &\textup{ if } f(3)\not= 0, \, f' (3)  \not= 0, \,  \textup{ and } \bigl(\chi(f'(3)\cdot f +2f(3)\cdot  f') \bigr) (3)=1,\\
3^{16} &\textup{ if } f(3)\not= 0, \, f' (3)  \not= 0, \,  \textup{ and } \bigl(\chi(f'(3)\cdot f +2f(3)\cdot  f') \bigr) (3) \in \{j, j^2\}.
\end{cases}
$$ 
\end{definition}

We give another definition
\begin{definition} 
\label{defmu(d)}
Let $f, f'\in V$ and let $d\in \N_3$.  We denote by $\mu (f, f', d)$ the positive integer defined by
$$
\mu (f, f', d) :=
\begin{cases}
3^{ {12}} & \textup{ if } 3\mid d, \ f(3)=f' (3) =0,\\
\mu( f, f')& \textup{ otherwise.}
\end{cases}
$$
\end{definition}

\subsection{The crucial sum} 
\label{crucialsum}
For positive integers $d$ and $a$, recall that ${\rm free} (d, a)$ is the largest squarefree integer dividing $d$ and coprime with $a$. In other words, we have
$$
{\rm free} (d, a)= \prod_{p \mid d\atop p\nmid a} p,
$$
which simplifies to ${\rm free }(d, a) =  d /(d, a)$, when $d$ is squarefree.
 
For $f$ and $f'\in V$ and $d \in \N_3$, we introduce the integer
\begin{equation}
\label{defD}
D(d,f, f') := \Delta (f)^6 \  {\rm free}\bigl( \Delta (f'), \Delta (f)\bigr)^4 \, \mu (f, f', d),
\end{equation}
the set 
$$
\mathcal S( f, f') := \bigl\{ d\in \N_3 : \bigl(d, \Delta (f) \Delta (f')\bigr)=1\bigr\},
$$
and the associated summatory function
\begin{equation}
\label{definitionS3}
S (X, f, f') := \sum_{d} 2^{\omega_{3}^* (d)},
\end{equation}
where the sum is over 
\begin{equation*}
d \in \mathcal S (f, f' ) \text{ and } {\rm free} (d,3)\leq \Bigl(\, X \Big/D (d, f, f' )\, \Bigr)^{1/6}. 
\end{equation*}
Gathering the above notations, we define the crucial sum ${\rm Heis}(X, 3)$ announced in \S \ref{Thegeneral}.

\begin{definition}
\label{definitionHeis} 
For $X \geq 2$ and the prime $\ell =3$, the associated {\it Heisenberg sum} ${\rm Heis} (X, 3)$ is
$$
{\rm Heis} (X, 3) := 2^{-2} 3^{-3}\, \underset{ f, f'\in V\atop f,\, f' \textup{\, lin. indep.}}{\sum \ \sum} 3^{ \vert\, {\rm supp}_3 f \, \cup \, {\rm supp}_3\, f'\vert}\cdot {\mathbbm  1} (f, f') \cdot S(X, f, f').
$$ 
\end{definition}

\noindent It is an exercise to verify that Definition \ref{definitionHeis} does not dependent on the way we have chosen the characters $\chi_p$ of order $3$ for each $p \in \mathbb{P}$. Combining Theorem \ref{tHDisc} (with $\ell = 3$) and equation (\ref{eExpandIndicator}), we obtain

\begin{proposition}
\label{pAnaHeisSum}
We have for every $X \geq 2$ the equality
\[
N(\textup{Heis}_3, X) = \textup{Heis}(X, 3).
\]
\end{proposition}

\noindent To state our main result we introduce the following notations 
\begin{itemize}
\item $\mathbbm 1_{ \{3\} }$ is the characteristic function of the set $\{ 3\}$,
\item $\psi_3$ is the multiplicative function defined on squarefree integers, satisfying 
\begin{equation}
\label{Defpsi3}
\psi_3 (p) = p/(p+2)
\end{equation}
(see the general definition given in \eqref{defpsi*}),

\item $\lambda$ is the multiplicative function defined on squarefree integers, satisfying
\begin{equation*}
\lambda (p)= \bigl( 1+2 /(p^{1/2} (p+2)\bigl)^{-1},
\end{equation*} 

\item $\alpha_3$ is the infinite product
\begin{equation}
\label{Defalpha3}
\alpha_3:= \frac 34 \prod_p \Bigl\{\, \Bigl( 1+\frac 1p + \frac {(\frac p3)}p\, 
\Bigr)\cdot \Bigl( 1 - \frac 1p\Bigr)
\Bigr\} 
\end{equation}
(see the general definition given in \eqref{Cell*}),

\item $H_0$ is the constant defined by\footnote{In $H_0$ we are summing over all primitive Dirichlet characters with order $3$ and with squarefree conductor $\Delta > 1$ coprime to $3$, while in the sum $H_2$ we are summing over all primitive Dirichlet characters with order $3$ and with conductor $9 \Delta$, where $\Delta \geq 1$ is squarefree and coprime to $3$.}
\begin{multline}
\label{defH0}
H_0:=  \sum_{\Delta \in \N_3^*\atop \Delta >1} \lambda  (\Delta)\, \psi_3 (\Delta)\cdot \frac{  {3}^{\omega (\Delta)}}{\Delta^{3/2}} 
\sum_{f\in V^* (\Delta)} \\ \Bigr\{\, \prod_{p\in \PP_3^*} \, \Bigl( \, 1 +2 \frac{ \chi (f) (p) +\chi (2f) (p)}{p+2}+\frac 2{p^{1/2}(p+2)} \Bigr)\,\Bigr\},
\end{multline}

\item $H_1$ is the constant defined by
\begin{multline}
\label{defH1}
H_1 :=  \sum_{\Delta \in \N_3^*\atop \Delta > 1} \lambda  (\Delta)\, \psi_3 (\Delta) \cdot \frac{{3}^{\omega (\Delta)}}{\Delta^{3/2}} 
\sum_{f\in V^* (\Delta)\atop \chi (f) (3) =1} \\
\Bigr\{\, \prod_{p\in \PP_3^*} \, \Bigl( \, 1 +2 \frac{ \chi (f) (p) +\chi (2f) (p)}{p+2}+\frac 2{p^{1/2}(p+2)} \Bigr)\,\Bigr\},  
\end{multline}

\item $H_2$ is the constant defined by 
\begin{multline}
\label{defH2}
H_2:=  \sum_{\Delta \in \N_3^*\atop \Delta \geq 1} \lambda  (\Delta)\, \psi_3 (\Delta)\cdot \frac{  {3}^{\omega (\Delta)}}{\Delta^{3/2}} 
\sum_{f\in V^* (\Delta)} \\ 
\sum_{\eta =1,2} \Bigr\{\, \prod_{p\in \PP_3^*} \, \Bigl( \, 1 +2 \frac{ \chi (f+\eta \mathbbm 1_{\{3\}}) (p) +\chi (2f +2 \eta \mathbbm 1_{\{3\}}) (p)}{p+2}+\frac 2{p^{1/2}(p+2)} \Bigr)\,\Bigr\}.
\end{multline}
\end{itemize} 

\noindent We now have all the tools to define the constant
\begin{equation}
\label{defcoeff}
c({\rm Heis}_3) := 2^{-2} \Bigl( \, \frac{{32}}{3^6}\cdot  H_0 +  \frac 8{3^6} \cdot H_1 +\frac{10}{3^7}\cdot  H_2\, \Bigr) \, \alpha_3.
\end{equation}
We will prove the following theorem, which combined with Proposition \ref{pAnaHeisSum} gives Theorem \ref{tMain}.

\begin{theorem} 
\label{GreatTheorem} 
Uniformly for $X\geq 2$, we have the equality
$$
{\rm Heis} (X, 3) = c({\rm Heis}_3) \cdot X^{1/4} \bigl( 1+ O \bigl(\,(\log X)^{-1}\bigr)\, \bigr).
$$
\end{theorem}

By utilizing the full strength of the Siegel--Walfisz Theorem one can improve the above error term to $O_A \bigl( \,(\log X)^{-A}\,\bigr)$ where $A > 0$ is arbitrary.

\begin{remark}
\label{rComparison}
In \S \ref{ssComments}, we will prove that the Euler product appearing in the definition of $H_0$ is essentially the product of the square of the modulus of cubic $L$--functions at the point $1$, see equation (\ref{ePf}). This leads to the observation that the constant $c({\rm Heis}_3)$ has obvious similarities with the constant $c (D_4)$, the value of which is given in Theorem \ref{tDisc}. These two constants are defined as series of values of Dirichlet $L$--functions at the point $1$. In the case of $c(D_4)$ the associated characters have order $2$, in the case of $c({\rm Heis}_3)$ this order is $3$.
\end{remark}
 
\subsection{The archetypical sum}
We first consider the subsum ${\rm Heis}^*(X)$ defined by\footnote{From now on, many notations will be shortened by omitting the dependency on the prime $\ell=3$.}
\begin{equation*}
{\rm Heis}^* (X) := 2^{-2} 3^{-3}\, \underset{ f, f'\in V^*\atop f,\, f' \text{\, lin. indep.}}{\sum \ \sum} 3^{ \vert\, {\rm supp}_3 f \, \cup \, {\rm supp}_3\, f'\vert}\cdot {\mathbbm 1} (f, f') \cdot {S}^*(X, f, f'),
\end{equation*}
where
\begin{itemize}
\item $V^* $ is defined in \eqref{defV*}, 
\item $  {S}^*  (X, f, f')$ is the subsum of  ${S}  (X, f, f')$, where we exclude all the $d$ divisible by $3$ (see \eqref{definitionS3}). 
\end{itemize}
Note that the subsum ${\rm Heis}^*(X)$ contains exactly those terms from ${\rm Heis} (X, 3)$ with $\mu(f, f', d) = 1$. Algebraically, this subsum corresponds to nonic Heisenberg extensions unramified at $3$. This is a convenient first sum to consider, since it avoids the many case distinctions in the definition of the function $\mu(f, f')$. We have the equality
\begin{equation}
\label{expressionHeis}
{\rm Heis}^* (X) = 2^{-2} 3^{-3}\, \underset{ f, f'\in V^*\atop f,\, f' \text{\, lin. indep.}}{\sum \ \sum} 3^{ \vert\, {\rm supp}  f \, \cup \, {\rm supp} \, f'\vert}\cdot {\mathbbm 1} (f, f') \cdot \Bigl( \sum_{d}2^{\omega (d)}\Bigr),
\end{equation}
where $d$ satisfies the following conditions
\begin{equation}
\label{summation1}
\begin{cases}
d\in \N_{3}^*,\\
\bigl(d, \Delta (f) \Delta (f')\, \bigr)=1, \\ 
1 \leq d \leq X^{1/6} \Delta (f)^{-1} \, \Delta (f')^{-2/3}\, \bigl( \, \Delta (f), \Delta (f')\, \bigr)^{2/3}.
\end{cases}
\end{equation}
Let
\begin{multline}
\label{defc*}
C_{\rm Heis^*} := 2^{-2} 3^{-3} \alpha_3 \sum_{\Delta \in \N_3^*\atop \Delta >1} \psi_3 (\Delta)\cdot \frac{3^{\omega (\Delta)}}{\Delta^{3/2}} 
\hspace{-0.2cm} \sum_{f\in V^* (\Delta)} \Bigr\{\, \prod_{p\in \PP_3^*} \, \Bigl( \, 1 +2 \frac{ \chi (f) (p) +\chi (2f) (p)}{p+2}\, \Bigr)\,\Bigr\}\\
\times 
\Bigl\{ \, \prod_{p \in \PP_3^*\atop p \nmid \Delta} \Big( 1 + \frac 2{p^{1/2}\bigl(\, p+2 (1+\chi (f) (p) +\chi (2f) (p))\, \bigr)} \Bigr)\, 
\Bigr\}
\end{multline}
where $\alpha_3$ and $\psi_3 $ are defined in \eqref{Defalpha3} and in \eqref{Defpsi3}. Thanks to \eqref{chi(f)=0or} and easy transformations, $C_{\rm Heis^*}$ can also be written as
$$ 
C_{\rm Heis^*} := 2^{-2} 3^{-3} \alpha_3  H_0,
$$
with $H_0$ defined in \eqref{defH0}. We will prove the following

\begin{proposition} 
\label{archetype}
Uniformly for $X\geq 2$ one has the equality
$$
{\rm Heis^*} (X) = C_{\rm Heis^*}  \cdot X^{1/4}  +O \bigl(\ X^{1/4}(\log X)^{-1}\, \bigr).
$$
\end{proposition}

We will prove in Proposition \ref{1289*} that $C_{\rm Heis^*}$ is positive, which implies that the above formula is an asymptotic one.

\subsection{The other sums} 
\label{theothersums*}
The subsum ${\rm Heis^*} (X)$ will be a model to treat the other subsums constituting ${\rm Heis } (X, 3)$. According to the definition of the $\mu$--functions, it is natural to consider the following fourteen subsums of 
${\rm Heis } (X, 3)$, denoted by ${\rm Heis}^{\eqref{C1}} (X)$, ${\rm Heis}^{\eqref{C2}} (X)$, ${\rm Heis}^{\eqref{C3}}$, ..., ${\rm Heis}^{\eqref{C14}} (X)$
where the exponent of ${\rm Heis}$ corresponds to the additional restrictions imposed to the variables of summation $d$ in $S (X,f, f')$ and to the pair $(f, f')$ in the first double summation in the Definition~\ref{definitionHeis}:
\begin{equation}
\label{C1}
3\nmid d, \, f(3)=f'(3) =0,
\end{equation}
\begin{equation}
\label{C2}
3\nmid d, \, f(3)=0, \, f'(3) \not=0,\, \chi (f) (3)=1,
\end{equation}
\begin{equation}
\label{C3}
3\nmid d, \, f(3)=0, \, f'(3) \not=0,\, \chi (f) (3) \in \{j, j^2\},
\end{equation}
\begin{equation}
\label{C4}
3\nmid d, \, f(3)\not=0, \, f'(3) =0,\, \chi (f') (3)=1,
\end{equation}
\begin{equation}
\label{C5}
3\nmid d, \, f(3)\not= 0, \, f'(3) =0,\, \chi (f') (3) \in \{j, j^2\},
\end{equation}
\begin{equation}
\label{C6}
3\nmid d, \, f(3)\not=0, \, f'(3) \not=0,\,  \bigl(\chi(f'(3) \cdot f +2f(3) \cdot f') \bigr) (3)=1,
\end{equation}
\begin{equation}
\label{C7}
3\nmid d, \, f(3)\not= 0, \, f'(3) \not=0,\,  \bigl(\chi(f'(3) \cdot f +2f(3) \cdot f') \bigr) (3) \in \{j, j^2\},
\end{equation}
\begin{equation}
\label{C8}
3\mid d, \, f(3)=f'(3) =0,
\end{equation}
\begin{equation}
\label{C9}
3\mid d, \, f(3)=0, \, f'(3) \not=0,\, \chi (f) (3)=1,
\end{equation}
\begin{equation}
\label{C10}
3\mid d, \, f(3)=0, \, f'(3) \not=0,\, \chi (f) (3) \in \{j, j^2\},
\end{equation}
\begin{equation}
\label{C11}
3\mid d, \, f(3)\not=0, \, f'(3) =0,\, \chi (f') (3)=1,
\end{equation}
\begin{equation}
\label{C12}
3\mid d, \, f(3)\not= 0, \, f'(3) =0,\, \chi (f') (3) \in \{j, j^2\},
\end{equation}
\begin{equation}
\label{C13}
3\mid d, \, f(3)\not=0, \, f'(3) \not=0,\,  \bigl(\chi(f'(3 )\cdot f +2f(3) \cdot f') \bigr) (3)=1,
\end{equation}
\begin{equation}
\label{C14}
3\mid d, \, f(3)\not= 0, \, f'(3) \not=0,\,  \bigl(\chi(f'(3) \cdot f +2f(3) \cdot f') \bigr) (3) \in \{j, j^2\}.
\end{equation}  
In each of these cases, the factor  $\mu (d,f, f')$ is constant. We have the obvious equalities
$$
{\rm Heis}^* (X) = {\rm Heis}^{\eqref{C1} }(X),
$$
and
\begin{equation}
\label{Heis=Heis+Heis}
{\rm Heis} (X, 3)={\rm Heis}^{\eqref{C1}}(X)+{\rm Heis}^{\eqref{C2}}+ \cdots + {\rm Heis}^{\eqref{C14}}(X).
\end{equation}
 
By following the proof of Proposition \ref{archetype} and by indicating the alterations between the different cases, we will prove in \S \ref{theothersums}

\begin{proposition}
\label{allthecases} 
Let $(i, j) = $ {\rm \eqref{C1}}, {\rm  \eqref{C2}},  {\rm \eqref{C3}}, \dots, or {\rm \eqref{C14}}. Then there exists a constant $C^{(i, j)} > 0$ such that
$$
{\rm Heis}^{(i, j)} (X )= 2^{-2} 3^{-3} \alpha_3 \,C^{(i, j)} X^{1/4} \bigl( 1+ O \bigl(\,(\log X)^{-1}\, \bigr)\, \bigr).
$$
Furthermore, we have the equalities
\begin{equation*}
\begin{matrix}
C^{\eqref{C1}} &= &H_0, & C^{\eqref{C8}}&=& 3^{ {-3}}\cdot H_0 ,\\
C^{\eqref{C2}} &= & 2\cdot 3^{-2}\cdot H_1, & C^{\eqref{C9}}&=&  2\cdot 3^{-2}\cdot H_1,\\
C^{\eqref{C3}} &= & 2\cdot 3^{-3} \cdot (H_0-H_1), & C^{\eqref{C10}}&=& 2\cdot 3^{-3} \cdot (H_0-H_1),\\
C^{\eqref{C4}} &= &  3^{-4}\cdot H_2  , & C^{\eqref{C11}}&=&    3^{-4}\cdot H_2  ,\\
C^{\eqref{C5}} &= &2\cdot 3^{-5}\cdot H_2 , & C^{\eqref{C12}}&=& 2\cdot 3^{-5}\cdot H_2 ,\\
C^{\eqref{C6}} &=& 2\cdot 3^{-4} \cdot H_2 , & C^{\eqref{C13}}&=&  2\cdot 3^{-4} \cdot H_2 ,\\
C^{\eqref{C7}}& =& 4\cdot 3^{-5}\cdot H_2 , & C^{\eqref{C14}}&=&   4\cdot 3^{-5}\cdot H_2  .\\
\end{matrix}
\end{equation*}
\end{proposition}

Gathering the decomposition given by \eqref{Heis=Heis+Heis} and the explicit values given by Proposition \ref{allthecases}, we complete the proof of Theorem \ref{GreatTheorem} through the equality
$$
c({\rm Heis}_3) = 2^{-2} 3^{-3}\, \alpha_3 \bigl( C^{\eqref{C1}} + \cdots + C^{\eqref{C14}}\bigr),
$$
which gives the explicit value announced in \eqref{defcoeff}. The inequality $c({\rm Heis}_3) > 0$ is a consequence of the inequalities $H_0 > 0$ (see Proposition \ref{1289*} below) and of the trivial inequality
$$
{\rm Heis} (X,3)\geq {\rm Heis}^* (X),
$$
since every subsum ${\rm Heis}^{\eqref{C2}}(X), \dots, {\rm Heis}^{\eqref{C14}} (X)$ is non-negative.
 
\section{Study of the archetypical sum}  
In this section we will prove Proposition \ref{archetype} concerning the sum ${\rm Heis}^* (X)$ as it appears in \eqref{expressionHeis} with the conditions of summation \eqref{summation1}.

\subsection{Trivial bounds and restrictions}
The number of positive divisors of the integer $n \geq 1$ is denoted  by  $\tau (n)$  and for $X\geq 1$, we write 
$$
\LL := \log 2X.
$$
In the course of the statements or proofs, the reader will find constants $A_0, A_1, \dots$ (particularly as exponents of $\LL$) for which it is possible to give explicit values, but we will refrain from doing so.

\subsubsection{Classical lemmas from analytic number theory}  
We will use the following bounds.

\begin{lemma}
\label{sumbomega(n)}
Let $b > 0$ be given. Then uniformly for $X \geq 1$ one has 
$$
\sum_{n\leq X} b^{\omega (n)} = O (X \LL^{b-1}) \textup{ and } \sum_{n\leq X\atop n \in \N_{3}^*} b^{\omega (n)} = O (X \LL^{b/2-1}) 
$$
\end{lemma}

The following lemma shows that in the sums we will meet, the contribution of the integers with a huge number of prime factors is small.

\begin{lemma}
\label{manyprimefactors} 
Let $b$ and $b'>0$ be given. Then there exists $B_0 = B_0(b, b')$ such that uniformly for $X \geq 1$ one has
$$
\sum_{n \leq X\atop \omega (n) > B_0 \log \log X} b^{\omega (n)} = O \bigl(\,  X \LL^{-b'} \bigr).
$$
\end{lemma}

\begin{proof} 
Let $\mathcal E_{B_0} (X)$ be the the set of integers $n\leq X$ such that $\omega (n) >  B_0 \log \log X$. We trivially have
$$
\vert \mathcal E_{B_0} (X)\vert \cdot 2^{B_0 \log \log X} \leq \sum_{n\leq X} \tau (n) \sim X \LL,
$$
which gives the bound $\vert \mathcal E_{B_0} (X) \vert \ll X\, \LL^{1-B_0 \log 2}$. Now, by the Cauchy--Schwarz inequality and by the first bound given by Lemma \ref{sumbomega(n)}, we have the inequalities 
$$
\sum_{n \leq X\atop \omega (n) > B_0 \log \log X} b^{\omega (n)}\ll  \vert \mathcal E_{B_0} (X)\vert^{1/2} \Bigl( \sum_{n\leq X} b^{2\omega (n)}\Bigr)^{1/2}\ll 
X \LL^{b^2/2 - (B_0 \log 2)/2}, 
$$
which is $\ll X \LL^{-b'}$ with the choice $B_0 = (b^2 + 2b')/\log 2$.
\end{proof}

\subsubsection{A trivial bound for ${\rm Heis} ^*( X)$} 
We first consider the sum (see \eqref{expressionHeis})
\begin{equation*}
S^*(X,f, f') = \sum_d 2^{\omega (d)},
\end{equation*}
where the integer $d$ satisfies the conditions \eqref{summation1}. The last condition of \eqref{summation1} implies the inequality 
\begin{equation}
\label{214}
\Delta (f)  \, \Delta(f')^{2/3} \, (\Delta (f), \Delta (f'))^{-2/3} \leq X^{1/6},
\end{equation}
which also implies
\begin{equation}
\label{restrictionDD'}
\Delta (f) \leq X^{1/6} \text{ and }  \Delta (f') \leq X^{1/4}.
\end{equation} 
A direct application of the second part of Lemma \ref{sumbomega(n)} leads to the bound
\begin{equation}
\label{S3*<<}
S ^* (X, f, f') \ll X^{1/6} \Delta (f) ^{-1} \, \Delta (f')^{-2/3}\, \bigl(\Delta (f) , \Delta (f')\bigr)^{2/3}.
\end{equation}
Later, in this paper, we will give a more precise formula for this quantity (see Proposition \ref{complexanalysis*} below).

We insert the bound \eqref{S3*<<} into \eqref{expressionHeis}. However, given  $\Delta \in \N_{3}^*$, there are $2^{\omega (\Delta)}=2^{\vert\, {\rm supp\,} f\, \vert }$ functions $f \in V^*$ such that $\Delta (f) = \Delta$. These remarks and Lemma \ref{1leq1} lead to the bound 
\begin{equation}
\label{shortcut}
{\rm Heis}^* (X)
\ll  X^{1/6} \ 
\underset{ \Delta,\  \Delta'}{\sum \sum } \, 3^{\omega (\Delta \,\Delta')}\cdot 2^{\omega (\Delta)}\cdot 2^{\omega (\Delta')}  \, \Delta^{-1} \, \Delta'^{-2/3}\, (\Delta , \ \Delta')^{2/3},
\end{equation}
where $\Delta$ and $ \Delta'$ belong to $\N_{3}^*$ and satisfy \eqref{214}.

To study this sum, we put $\gamma = (\Delta, \Delta')$, $\Delta = \gamma \delta$ and $\Delta' = \gamma \delta'$ to write the inequality 
\begin{equation}
\label{234}
{\rm Heis}^* (X)
\ll  X^{1/6} \sum_{\gamma} 12^{\omega (\gamma)}\gamma^{-1}  \sum_{\delta} 6^{\omega (\delta)}\delta^{-1} \sum_{\delta'} 6^{\omega (\delta')} \delta'^{-2/3}.
\end{equation}
By a repeated application of Lemma \ref{sumbomega(n)}, by partial summations and by the crude inequalities \eqref{restrictionDD'}, we arrive at the inequality
\begin{equation}
\label{generalboundforHeis}
{\rm Heis}^* (X) \ll X^{1/4} \LL^5.
\end{equation}
This trivial bound just misses the expected order of magnitude of ${\rm Heis}^* (X)$  announced in Proposition \ref{archetype} by a power of $\LL$.

\subsubsection{Restriction on the size of $\Delta (f)$} 
Let $\Delta_0 >1$ be given. We denote by ${\rm Heis}^* (X; \Delta  >\Delta_0)$ the subsum of ${\rm Heis}^* (X)$ corresponding to the following restrictions of summations over $f$ and $f'$ (compare with the conditions in \eqref{expressionHeis})
\begin{equation}
\label{restriction}
\begin{cases}
f, \ f' \in V^*,\\
f,\, f' \text{ linearly  independent},\\
\Delta (f) > \Delta_0.
\end{cases}
\end{equation}
We will prove the following

\begin{proposition}
\label{largeDelta} 
There exists $A_0 > 0$ such that, uniformly for $X\geq 2$, one has the upper bound
$$
{\rm Heis}^* (X; \Delta > \LL^{A_0}) \ll X^{1/4}\LL^{-1}.
$$
\end{proposition}

\begin{proof} 
By a computation similar to \eqref{234}, one has the inequality
\begin{equation*}
{\rm Heis}^* (X ;\Delta >\Delta_0)
\ll  X^{1/6} \sum_{\gamma} 12^{\omega (\gamma)}\gamma^{-1}  \sum_{\delta} 6^{\omega (\delta)}\delta^{-1} \sum_{\delta'} 6^{\omega (\delta')} \delta'^{-2/3}, 
\end{equation*}
where the sum is  over the triples of positive integers $(\gamma, \delta, \delta')$ such that 
$$
\begin{cases}
\gamma\, \delta > \Delta_0,\\
\gamma\, \delta\,  \delta'^{2/3} \leq X^{1/6},
\end{cases}
$$
(see \eqref{214} for the last condition). Summing first over $\delta'$ we get, for some constant $A_1 >0$, the bound
\begin{align*}
{\rm Heis}^* (X; \Delta >\Delta_0)
& \ll  X^{1/4}  \LL^{A_1}\,  \sum_{\gamma} 12^{\omega (\gamma)}\gamma^{-3/2}  \sum_{\delta} 6^{\omega (\delta)}\delta^{-3/2},\\
& \ll X^{1/4}\LL^{A_1}\, \sum_{\Delta > \Delta_0} 18^{\omega (\Delta)}\,  \Delta^{-3/2}, 
 \end{align*}
since $\gamma \delta = \Delta$. If we choose $\Delta_0 = \LL^{A_0}$ for a sufficiently large value of $A_0$, Lemma \ref{sumbomega(n)} and partial summation show that the above expression is $\ll X^{1/4}\LL^{-1}$.
\end{proof}

\subsubsection{Restriction on the size of $\Delta (f')$} 
In this paragraph, we show that we can restrict ourselves to large values of
$\Delta (f')$ which means $\Delta (f') > X^{1/4}\LL^{-A_2}$. 

To be more precise, let $A_0$ be as in Proposition \ref{largeDelta}. For $\Delta'_0 > 1$ let 
$$
{\rm Heis^*} (X ; \Delta \leq \LL^{A_0}, \Delta' < \Delta'_0)
$$
be the subsum of ${\rm Heis}^* (X)$ corresponding to the restriction of summations (compare with \eqref{expressionHeis} and with \eqref{restriction})
\begin{equation}
\label{restriction10}
\begin{cases}
f, \ f' \in V^*,\\
f,\, f' \text{ linearly independent,}\\
\Delta (f) \leq  \LL^{A_0},\\
\Delta (f') < \Delta'_0.
\end{cases}
\end{equation}
We will prove

\begin{proposition}
\label{smallDelta'} 
Let $A_0$ be as in Proposition \ref{largeDelta}. There exists $A_2 >0$ such that, uniformly for $X\geq 2$, one has the upper bound 
$$
{\rm Heis}^* \bigl(X ;\Delta \leq \LL^{A_0}, \Delta' <X^{1/4}\LL^{-A_2}\bigr)  \ll X^{1/4}\LL^{-1}.
$$
\end{proposition}

\begin{proof} 
The proof mimics the proof of the crude bound \eqref{generalboundforHeis}. It suffices to replace the conditions \eqref{restrictionDD'} by the two  present hypotheses: $\Delta (f) \leq \LL^{A_0}$ and $\Delta (f') < X^{1/4} \LL^{-A_2}$ and to choose
$A_2$ sufficiently large to replace the exponent $5$ by $-1$ on the right--hand side of \eqref{generalboundforHeis}.
\end{proof}
 
\subsubsection{Restriction on the number of prime factors of $\Delta (f')$} 
Thanks to Propositions \ref{largeDelta} and \ref{smallDelta'}, it remains to study  the contribution of the pairs $(f, f')\in V^*\times  V^*$, linearly independent, with $ \Delta (f) $ small (which means $\leq \LL^{A_0}$) and with $\Delta (f')$ of size almost maximal (which means between $X^{1/4} \LL^{-A_2}$ and $X^{1/4}$). We continue our preparation of the pairs $(f, f')$ by controlling the number of prime factors of $ \Delta (f')$. Let $A_0$ and $A_2$ be as in Propositions \ref{largeDelta} and \ref{smallDelta'}. Let $A_3 > 0$ to be fixed later. Let 
$$
{\rm Heis^*} \bigl(X ; 1<\Delta \leq   \LL^{A_0}, \Delta '\geq X^{1/4}\LL^{-A_2}, \omega (\Delta') \geq  A_3 \log \log X\, \bigr)
$$
be the subsum of 
${\rm Heis}^* (X)$ corresponding to the restriction of summations (compare with \eqref{expressionHeis} and \eqref{restriction10})
\begin{equation}
\label{lemon}
\begin{cases}
f, \ f' \in V^*,\\
1< \Delta (f) \leq  \LL^{A_0}, \\
\Delta (f') \geq X^{1/4} \LL^{-A_2},\\
\omega \bigl( \Delta (f')\bigr) \geq   A_3 \log \log X.
\end{cases}
\end{equation}

\begin{remark} 
The second and third condition of \eqref{lemon} imply that $f$ and $f'$ have distinct supports. So these functions are linearly independent, as soon as  $\Delta (f) > 1$.
\end{remark} 

We will prove

\begin{proposition}
\label{largeomega} 
Let $A_0$ and $A_2$ be as in Propositions \ref{largeDelta} and \ref{smallDelta'}. Then there exists $A_3$ such that, uniformly for $X \geq 2$, one has the upper bound 
$$
{\rm Heis}^* \bigl(\,X; 1< \Delta \leq \LL^{A_0}, \Delta' > X^{1/4} \LL^{-A_2}, \omega (\Delta ')\geq   A_3 \log \log X\, \bigr)  \ll X^{1/4}\LL^{-1}.
$$
\end{proposition}

\begin{proof} 
We go back to the inequality \eqref{shortcut} to perform a trivial summation over $\Delta \leq \LL^{A_0}$. Hence, for some $A_4$, we have the inequality
\begin{multline*}
{\rm Heis}^* \bigl(\,X; 1< \Delta \leq \LL^{A_0}, \Delta' >X^{1/4}\, \LL^{-A_2}, \omega (\Delta ')\geq  A_3 \log \log X\, \bigr) \\
\ll X^{1/6} \LL^{A_4}\ \sum_{\Delta' < X^{1/4}\atop \omega (\Delta') \geq  (A_3/2)\log \log X^{1/4} } 6^{\omega (\Delta')} \Delta'^{-2/3}
\ll X^{1/4} \LL^{-1},
\end{multline*}
by Lemma \ref{manyprimefactors}, by a partial summation and by choosing $A_3$ sufficiently large.
\end{proof}

We have finished with the technical preparation of $\Delta (f)$ and $\Delta (f')$. So it is natural to define the subsum ${\rm Heis}^\dag (X)$ of ${\rm Heis}^* (X)$, defined in \eqref{expressionHeis}, by imposing the following additional restrictions of summation on $f$ and $f'$
\begin{equation}
\label{restriction2}
\begin{cases}
f, \ f' \in V^*,\\
1< \Delta (f)) \leq  \LL^{A_0},\\
\Delta (f')\geq X^{1/4} \, \LL^{-A_2},\\
\omega (\Delta (f')) \leq   A_3 \log \log X, \\
\Delta  (f)\, \Delta (f')^{2/3} \, (\Delta (f) , \Delta (f'))^{-2/3} \leq X^{1/6},
\end{cases}
\end{equation}
where $A_0$, $A_2$ and $A_3$ are defined in Propositions \ref{largeDelta},  \ref{smallDelta'} and \ref{largeomega}. Gathering Propositions  \ref{largeDelta}, \ref{smallDelta'} and \ref{largeomega}, we see that the proof of Proposition \ref{archetype} is reduced  to the proof of the  formula
\begin{equation}
\label{asympforHeisdag}
{\rm Heis}^\dag (X) = C_{{\rm Heis}^*}\, X^{1/4} + O \bigl(X^{1/4}\LL^{-1} \bigr), 
\end{equation}
where $C_{{\rm Heis}^*}$ is defined in  \eqref{defc*} and where the $O$--constant is uniform for $X \geq 1$.

\subsection{\texorpdfstring{Inverting summations in ${\rm Heis}^\dag (X)$}{Inverting summations in Heis(X)}} 
We now benefit from the control of the sizes of the variables appearing in ${\rm Heis}^\dag (X)$ which is a subsum of ${\rm Heis}^* (X)$. By the last line of \eqref{summation1} and by the second and third lines of \eqref{restriction2} we see that $d$ satisfies the inequalities
$$
1 \leq d \leq X^{1/6} \Delta (f)^{-1}\,   \bigl(\, X^{1/4} \LL^{-A_2}\, \bigr)^{-2/3} \Delta (f)^{2/3} \leq \LL^{2A_2/3 } = \LL^{A_4},
$$
by definition. This means that the variable $d$ is almost constant and it is wise to  perform the summation over this variable at the very end of the proof. We decompose ${\rm Heis}^\dag (X)$ as
\begin{equation}
\label{sumoverd}
{\rm Heis}^\dag (X)=
\sum_{d\in\N_3^*\atop d \leq \LL^{A_4}} 2^{\omega (d)}\, U (X,d)
\end{equation}
with
\begin{equation}
\label{defU(d,X)}
U(X,d) = 2^{-2} 3^{-3}\, \underset{ f, \ f' }{\sum \ \sum} \ 3^{ \vert\, {\rm supp}  f \, \cup \, {\rm supp} \, f'\vert  } \cdot {\mathbbm 1} (f, f'), 
\end{equation}
where the pair of functions $(f, f')$ satisfies \eqref{restriction2}, the inequality
\begin{equation}
\label{<X1/6/d}
\Delta (f) \, \Delta (f')^{2/3} \, \bigl( \Delta (f),\Delta (f')\bigr)^{-2/3}\leq X^{1/6} d^{-1},
\end{equation}
which is a consequence of \eqref{summation1}, and the coprimality condition
$$
\bigl(\,d, \Delta (f) \Delta (f')\, \bigr) = 1.
$$

\subsection{\texorpdfstring{Factorisation of the function $\mathbbm 1 (f, f')$}{Factorisation of the function 1(f, f')}} 
To facilitate the study of the function $\mathbbm 1 (f, f')$, we put
$$
\mathcal E := {\rm supp\, } f \text{ and } \mathcal E' := {\rm supp\, }f'.
$$
In a unique way, we decompose $\mathcal E$ and $\mathcal E'$ as a disjoint union
\begin{equation}
\label{decompE*}
\mathcal E = \mathcal E_0 \cup \mathcal E_1 \text{ and }   \mathcal E = \mathcal E_0 \cup \mathcal E'_1,
\end{equation}
where, furthermore $\mathcal E_1$ and $\mathcal E'_1$ are disjoint. This decomposition incites to write the functions $f$ and $f'$ as
\begin{equation}
\label{decompf*}
f= f_0 \oplus f_1 \text{ and } f'= f'_0 \oplus f'_1,
\end{equation}
where ${\rm supp} f_0 = {\rm supp} f'_0 = \mathcal E_0$, ${\rm supp}\, f_1= \mathcal E_1$ and ${\rm supp} \, f'_1 = \mathcal E'_1$. We define
\begin{equation}
\label{decompDelta*}
\Delta_0 :=   \Delta(f_0) =  \Delta  (f'_0)= \prod_{p\in \mathcal E_0} p,
\end{equation}
and we define $\Delta_1$ and $\Delta'_1$ analogously. The integers $\Delta_0$, $\Delta_1$ and $\Delta'_1$ belong to $\N_{3}^*$ and are coprime in pairs. The numbers  $\Delta = \Delta_0 \Delta_1$ and $\Delta' = \Delta_0 \Delta'_1$  also belong to $\N_3^*$. We now start rewriting $\mathbbm 1 (f, f')$ in terms of characters.

\begin{lemma}
\label{expand1} 
Let $f, f' \in V^*$. We adopt the notations \eqref{decompE*}, \eqref{decompf*} and \eqref{decompDelta*}. We then have the equalities
\begin{multline}
\label{671} 
\underset{(z,z')\in \F_3^2 \atop  f(r) z +f'(r) z' =0}{\sum \sum }
\bigl( \chi (zf+z'f') \bigr) (r) = \\
1+
 \begin{cases} \chi (f'_0+f'_1) (r) +\chi (2(f'_0 +f'_1)) (r)  &\textup{ if } r\in \mathcal E_1,\\
 \chi (f_0+f_1) (r) +\chi (2(f_0+f_1)) (r)  &\textup{ if } r\in \mathcal E'_1,\\
\chi \bigl(f'_0 (r) (f_0+f_1)+ 2f_0(r)(f'_0+ f'_1) \bigr)(r)& \\
\qquad \quad+ \chi \bigl(2f'_0 (r) (f_0+f_1) + f_0 (r)(f'_0 +   f'_1)\,\bigr) (r)  &\textup{ if } r \in \mathcal E_0.\\
 \end{cases}
\end{multline}
\end{lemma}

\begin{proof} 
Solve the equation $f(r)z + f'(r) z' = 0$ in each of the three cases.
\end{proof} 

\begin{remark} 
Recall that the value of the right--hand side of \eqref{671} is $0$ or $3$.
\end{remark}

\subsection{\texorpdfstring{Decomposition of $U(X, d)$}{Decomposition of U(X, d)}} 
\label{decompU(X,d)}
We incorporate the decompositions \eqref{decompE*}, \eqref{decompf*} and \eqref{decompDelta*} in \eqref{defU(d,X)}. Combining Lemma \ref{expand1} with the notation introduced in \eqref{defVV*}  and with Definition \ref{definition1} we arrive at the equality
\begin{multline}
\label{sumsumsumsumsumsumsum*}
U(X,d)=
2^{-2} 3^{-3}\  \underset {\Delta_0,\ \Delta_1,\ \Delta'_1}{ \sum \ \sum \ \sum }\quad 
 \underset{\substack{f _0, \, f'_0\in V^* (\Delta_0)\\  f_1\in
V^* (\Delta_1)\\  f'_1 \in V^* (\Delta'_1)}}{ \sum\ \sum \ \sum\ \sum} \\
\prod_{r\mid \Delta_0}\Bigl\{ 1+ \chi \bigl(f'_0 (r) (f_0+f_1)+ 2f_0(r)(f'_0+ f'_1) \bigr)(r) 
+ \chi \bigl(2f'_0 (r) (f_0+f_1) + f_0 (r)(f'_0 +   f'_1)\,\bigr) (r)\Bigr\}\\
\times \prod_{r\mid \Delta_1}\bigl\{ 1+\chi (f'_0+f'_1) (r) +\chi (2(f'_0 +f'_1)) (r) \bigr\}
\prod_{r \mid \Delta'_1} \bigl\{ 1 +   \chi (f_0+f_1) (r) +\chi (2(f_0+f_1)) (r)    \bigr\},
\end{multline}
where the conditions of summation \eqref{restriction2} and \eqref{<X1/6/d} become
\begin{equation}
\label{finalconditions*}
\begin{cases}
\Delta_0, \, \Delta_1, \, \Delta'_1 \in \N_3^*,\\
(\Delta_0, \, \Delta_1) =(\Delta_0, \Delta'_1) = (\Delta_1, \Delta'_1) = (d, \Delta_0 \Delta_1 \Delta'_1) =1,  \\
1< \Delta_0 \Delta_1 \leq \LL^{A_0},\\
\Delta_0 \Delta'_1 \geq X^{1/4} \LL^{-A_2},\\
\omega \bigl(\Delta_0 \Delta'_1\bigr)\leq A_3 \log \log X,\\
\Delta_0\, \Delta_1\, {\Delta'_1}^{2/3} \leq X^{1/6}/d.
\end{cases}
\end{equation} 
In a condensed way, we write \eqref{sumsumsumsumsumsumsum*} as
\begin{equation*}
U (X,d) = 2^{-2} 3^{-3} \sum_{\boldsymbol \Delta}\quad  \sum_{\boldsymbol f}\ \prod_{r\mid \Delta_0}\{\cdots \} 
\prod_{r\mid \Delta_1}\{\cdots \} \prod_{r\mid \Delta'_1}\{\cdots \},
\end{equation*}
and we decompose $U(X,d)$ as
\begin{equation}
\label{U=MT+ET}
U(X,d) ={\rm MT} (X,d) + {\rm ET} (X,d),
\end{equation}
where 
\begin{equation}
\label{defMTdef*}
{\rm MT} (X,d):= 2^{-2} 3^{-3} \sum_{\boldsymbol \Delta}\quad  \sum_{\boldsymbol f}\  \prod_{r\mid \Delta'_1}\{\cdots \} ,
\end{equation}
and
\begin{equation}
\label{defETdef*}
{\rm ET} (X,d) := 2^{-2} 3^{-3} \sum_{\boldsymbol \Delta} \quad \sum_{\boldsymbol f}\Bigl(-1+  \prod_{r\mid \Delta_0}\{\cdots \} 
\prod_{r\mid \Delta_1}\{\cdots \} \Bigr)\,\prod_{r\mid \Delta'_1}\{\cdots \}.
\end{equation}
To describe the scenery of these sums we insist on the fact that $\Delta_0$ and $\Delta_1$ are very small variables. In contrast, $\Delta'_1$ is a large variable, and since $\Delta'_1$ has few prime divisors (see the fifth  line of \eqref{finalconditions*}), its largest prime divisor, that we will denote by $p_\infty := p_\infty (\Delta'_1)$, is also large. When summing over $p_\infty$, we will obtain cancellation between cubic characters as a consequence of a theorem of Siegel--Walfisz type (see Lemma \ref{sumchi(pi)**}). We will obtain Proposition \ref{ETissmall*} below, which shows that ${\rm ET} (X,d)$ is an error term. In the other direction, the term ${\rm MT}(X,d)$,  roughly speaking, appears to be the product of $X^{1/4}$ by a convergent series for which we will search for a concise value, which will lead to the value of $C_{{\rm Heis}^*}$ given in \eqref{defc*}. 

\subsection{\texorpdfstring{Study of ${\rm ET}(X, d)$}{Study of ET(X, d)}} 
\label{StudyofET}
We factorize $\Delta'_1$ as
\begin{equation}
\label{Delta'1=*}
\begin{cases}
\Delta'_1 =\Delta''_1 \, p_\infty,\\
p \mid \Delta''_1 \Rightarrow p < p_\infty.
\end{cases}
\end{equation}
Correspondingly, there are two possible decompositions of the function $f'_1$
\begin{equation}
\label{f'1=*}
f'_1 := f''_1 \oplus \mathbbm{1}_{p_\infty} \text{ or } f'_1=f''_1 \oplus 2\cdot \mathbbm{1}_{p_\infty},
\end{equation}
where $\Delta(f''_1) = \Delta''_1$, and $\mathbbm{1}_{p_\infty}$ is the characteristic function of the set $\{p_\infty\}$. We also have
$$
\chi (f'_1) = \chi(f''_1) \, \chi_{p_\infty} \text{ or } \chi(f_1') = \chi (f''_1) \, \chi_{p_\infty}^2,
$$
according to the cases listed in \eqref{f'1=*}. We return to \eqref{defETdef*} to highlight the summation over $p_\infty$:
\begin{multline}
\label{noteasy*}
{\rm ET}(X, d) \ll \underset{\Delta_0, \Delta_1, \Delta''_1} {\sum\, \sum \, \sum}\,    
 3^{  \omega (\Delta''_1)}\
  \underset{\substack{f _0, \, f'_0\in V^* (\Delta_0)\\  f_1\in
V^* (\Delta_1)\\  f''_1 \in V^* (\Delta''_1)}}{ \sum\ \sum \ \sum\ \sum} \\
\Bigl\vert
\sum_{p_\infty}   
\Bigl( -1 + \prod_{r\mid \Delta_0} \{\cdots\} \prod_{r\mid \Delta_1}\{\cdots\}\Bigr)\bigl( 1 +
\chi (f_0+f_1) (p_\infty) + \chi (2(f_0+f_1))(p_\infty)\bigr)\, 
\Bigr\vert\\ 
+ \textup{similar term},
\end{multline}
where, in the second line of \eqref{noteasy*}, we have chosen, for $f'_1$, the first   decomposition written in \eqref{f'1=*}. The {\it similar term} corresponds to the second decomposition in \eqref{f'1=*}. In \eqref{noteasy*}, the conditions of summation are deduced from \eqref{finalconditions*} by applying the decomposition  \eqref{Delta'1=*}.
 
We develop the product on the second line of \eqref{noteasy*} to bring out 
\begin{equation}
\label{LLA5*}
3(3^{\omega (\Delta_0 \Delta_1)}-1) \ (=O (\LL^{A_5}))
\end{equation}
products of cubic characters. This means that the sum over $p_\infty$ appearing in \eqref{noteasy*} is the sum of $O(\LL^{A_5})$ sums of the form
\begin{multline*}
A  (\boldsymbol \eta, \, \boldsymbol \zeta, \boldsymbol \epsilon, f_0, f'_0,f_1, f'_1)= 
\sum_{p_\infty}   
\prod_{r\mid \Delta_0}\Bigl\{\bigl[ \chi \bigl(f'_0 (r) (f_0+f_1)+ 2f_0(r)(f'_0+ f'_1) \bigr)(r) \bigr]^{\eta_{1r}}\\
\times \bigl[\,\chi \bigl(2f'_0 (r) (f_0+f_1) + f_0 (r)(f'_0 +   f'_1)\,\bigr) (r)\, \bigr]^{\eta_{2r}}\Bigr\}\\
\prod_{r\mid \Delta_1}\bigl\{\, \bigl[\chi (f'_0+f'_1) (r)\bigr]^{\zeta_{1r}}\cdot \bigl[\, \chi (2(f'_0 +f'_1)) (r)\, \bigr] ^{\zeta_{2r}} \bigr\}
 \\ \cdot 
\bigl[\,\chi(f_0+f_1)(p_\infty) \, \bigr]^{\epsilon_1}\cdot \bigl[\,\chi (2 (f_0+f_1))(p_\infty)\, \bigr]^{\epsilon_2}
\end{multline*}
where the exponents are non-negative integers and satisfy the inequalities
\begin{equation}
\label{conditionsofexponents*}
\begin{cases}
0 \leq \eta_{1r} + \eta_{2r} \leq 1\text{ for each } r \mid \Delta_0,\\
0\leq \zeta_{1r} + \zeta_{2r} \leq 1 \text{ for each } r \mid \Delta_1,\\
0\leq  \epsilon_1 +\epsilon_2 \leq 1,\\
\sum_{r \mid \Delta_0}(\eta_{1r} + \eta_{2r}) + \sum_{r \mid \Delta_1} ( \zeta_{1r} + \zeta_{r})  \geq 1.
\end{cases}
\end{equation}

We return to the definition of $\chi (f)$ given in \eqref{defchi(f)} and recall the equalities $f'_1 (p_\infty) = 1$ and $f_0 (p_\infty) = f'_0 (p_\infty) = f_1 (p_\infty) = 0$. Keeping only the terms depending on $p_\infty$, we get an equality
\begin{equation}
\label{A=A*}
\vert A  (\boldsymbol \eta, \, \boldsymbol \zeta, \boldsymbol \epsilon, f_0, f'_0,f_1, f'_1) \vert 
= \vert \widetilde A (\boldsymbol \eta, \, \boldsymbol \zeta, \boldsymbol \epsilon, f_0, f_1) \vert
\end{equation}
between moduli, where
\begin{align}
\widetilde A  (\boldsymbol \eta, \, \boldsymbol \zeta, \boldsymbol \epsilon, f_0, f_1)&= \sum_{p_\infty}   
\bigl[\, \chi (f_0+f_1) (p_\infty)\, \bigr]^{\epsilon_1 +2 \epsilon_2}\nonumber\\
&\qquad \times\prod_{r \mid \Delta_0}  \bigl[\, \chi_{p_\infty} (r) \bigr]^{f_0(r)(2\,\eta_{1r}+\eta_{2r} )}
\prod_{r \mid \Delta_1} \bigl[\, \chi_{p_\infty} (r)\, \bigr]^{\zeta_{1r} +2 \,\zeta_{2r}}\label{f0(r)=1=2}\\
&= \sum_{p_\infty}    \widetilde M (p_\infty),\label{deftildeM*}
\end{align}
by definition. Of course $p_\infty$ satisfies the conditions of summation deduced from \eqref{finalconditions*} by applying the factorization \eqref{Delta'1=*}. The exponent $f_0(r)$ appearing in \eqref{f0(r)=1=2} can take the value $1$ or $2 \bmod 3$. If its value is $2$, we have the equality 
$$
f_0 (r) (2\, \eta_{1r} + \eta_{2r}) = 2\, \eta_{2r} + \eta_{1r}.
$$
So in that case, we  can invert the roles of  $\eta_{1r}$ and $\eta_{2r}$ without affecting the conditions \eqref{conditionsofexponents*}. So we can always suppose that $f_0(r)=1$ in the definition of $\widetilde M (p_\infty)$. We also replace $f_0\oplus f_1$ by $f$ (see \eqref{decompf*}) and $\Delta_0  \Delta_1$ by $\Delta$ (see  \eqref{decompDelta*}). So $\widetilde M (p_\infty)$ equals
\begin{equation}
\label{concisedefinition*}
\widetilde M (p_\infty) = \bigl[\, \chi (f) (p_\infty)\, \bigr]^{\epsilon_1 +2 \epsilon_2}\, 
\prod_{r \mid \Delta} \bigl[\, \chi_{p_\infty} (r)\, \bigr]^{e_{1r} +2 e_{2r}},
\end{equation}
where $p_\infty \in \PP_{3}^*$ does not divide $\Delta$ and where the non-negative exponents $\epsilon_i$ and $e_{ir}$ satisfy
\begin{equation}
\label{812*}
\begin{cases}
0 \leq e_{1r}+e_{2r} \leq 1, \text{ for all } r\mid \Delta,\\
0 \leq \epsilon_1+ \epsilon_2 \leq 1\\
\sum_{r \mid \Delta}(e_{1r} + e_{2r})    \geq 1.
\end{cases}
\end{equation}
To obtain the desired cancellation when summing over $p_\infty$, we will show that $\widetilde{M}$ is a character of $\Z [j]$. As a first step we use the following
\begin{lemma}
\label{reciprocityapplied*} 
For every distinct primes  $p$ and $r$ in $\PP_3^*$, decomposed in the standard way:
$p =\pi \cdot \overline \pi$ and $r= \rho \cdot \overline \rho$, we have the equality
$$
\chi_p (r) = \overline{\chi_r (p)} \, \Bigl( \frac \pi \rho\Bigr)_3^2.
$$
\end{lemma}

\begin{proof} 
Combine the multiplicative properties of the cubic character, the cubic reciprocity law $(\pi /\rho)_3 = (\rho/\pi)_3$ (see \cite[Theorem1 p. 114]{Ir-Ro}, for instance) and the conjugation property $\overline{(\pi/\rho)_3} = (\overline \pi/ \overline \rho)_3$.
\end{proof}

We use Lemma \ref{reciprocityapplied*} to write the equality
$$
\chi_{p_\infty} (r) =\overline{ \chi_r (p_\infty)} \,\Bigl( \displaystyle{\frac {\pi_\infty} \rho }\Bigr)^2_3,
$$ 
where we decomposed in a standard way $p_\infty = \pi_\infty \cdot \overline{ \pi_\infty}$ and $r = \rho \cdot \overline{\rho}.$ 
 
Let $f\in V^*$, let $\boldsymbol \epsilon$ be a pair $(\epsilon_1, \epsilon_2)$ of positive integers and let $\boldsymbol e= (e_{1r}, e_{2r})_{r\in {\rm supp} \, f}$ be a $2\cdot \vert\,  {\rm supp}\, f\,\vert $--tuple of positive integers. Let $r\in \PP_3^*$ decomposed in the standard way $r =\rho \cdot \overline \rho$. For $z \in \Z [j]$, we define 
$$
M(z) = M(z, f, \boldsymbol \epsilon, \boldsymbol e):=
\bigl[ \, \chi(f) (z\overline z)\, \bigr]^{\epsilon_{1}} \ \bigl[\,\chi (2f)(z\overline z)\,\bigr]^{\epsilon_2}\ 
\prod_{r \in\,  {\rm supp}\, f} \Bigl[\,   \chi_r (z\overline z) \,\Bigl( \displaystyle{\frac z \rho }\Bigr)_3\, \Bigl]^{2e_{1r} +e_{2r}}.
$$

\begin{lemma} 
\label{nontrivial**}
Let $f\in V^*$ with a non-empty support. Let $p =\pi \cdot \overline \pi$ be the standard decomposition of a prime $p$ belonging to $\mathbb P_3^*$ but not to ${\rm supp} \, f $. We then have the equality
\begin{equation}
\label{tildeM=M(pi)*}
\widetilde {M} (p) = M (\pi).
\end{equation}
Suppose furthermore that the following conditions are satisfied:
\begin{equation*}
\begin{cases} 
0 \leq \epsilon_1 +\epsilon_2\leq 1,\\
0 \leq e_{1,r} + e_{2,r} \leq 1, \textup{ for each } r \in {\rm supp}\, f, \\
\sum_{r \in {\rm supp}\, f} (e_{1r} +e_{2r}) \geq 1.
\end{cases} 
\end{equation*}
Then the application $z\mapsto M(z)$ is a non-trivial multiplicative character over $\Z [j]$, with period dividing $\prod_{r\in {\rm supp}\, f} \rho$.
\end{lemma} 
 
\begin{proof} 
The equality \eqref{tildeM=M(pi)*} is a consequence of the construction of the function $M$ and Lemma \ref{reciprocityapplied*}.

For the second part, it is clear that $z \mapsto M(z)$ is a multiplicative character over $\Z[j]$, and it is also clear that its period divides $\prod_{r\in {\rm supp}\, f} \rho$. It remains to show that it is a non-trivial character.

Suppose that $M(z)$ is the trivial character. Note that $M(z)$ is a product
\[
\bigl[ \,  \prod_{r \in\,  {\rm supp}\, f} \chi_r^{f(r)}(z\overline z)\, \bigr]^{\epsilon_{1}} \ \bigl[\, \prod_{r \in\,  {\rm supp}\, f} \chi_r^{2f(r)} (z\overline z)\,\bigr]^{\epsilon_2}\ 
 \prod_{r \in\,  {\rm supp}\, f} \Bigl[\,   \chi_r (z\overline z) \,\Bigl( \displaystyle{\frac z \rho }\Bigr)\, \Bigl]^{2e_{1r} +e_{2r}},
\]
where all the factors have coprime period. Hence $M(z)$ trivial implies that
\[
\chi_r^{\epsilon_{1} f(r)}(z \overline{z}) \chi_r^{2\epsilon_{2} f(r)}(z \overline{z}) \Bigl[\,\chi_r (z\overline z) \left(\displaystyle{\frac z \rho }\right)\,\Bigr]\,^{2e_{1r} +e_{2r}}
\]
is the trivial character for every $r$ in the support of $f$. Now recall the inequalities
\[
0 \leq  \epsilon_1 +\epsilon_2\leq 1, \quad 0 \leq e_{1r} + e_{2r} \leq 1.
\]
But $\chi_r(z \overline{z})$ and $\left(\displaystyle{\frac z \rho }\right)$ are linearly independent characters. This forces
\[
\epsilon_1 = \epsilon_2 = e_{1r} = e_{2r} = 0,
\]
contrary to our third assumption.
\end{proof}  

\subsubsection{A Siegel--Walfisz type Theorem for standard primes}
The famous Siegel--Walfisz Theorem for rational primes gives equidistribution of primes $p\leq X$ in arithmetic progressions $a+kq$ (with $(a,q)=1$) uniformly for the modulus $q \leq \LL^A$ for any arbitrary given $A$. Such a phenomenom of equidistribution also holds for prime ideals in number fields since the associated $L$--functions have properties similar to those of Dirichlet $L$--functions. On that subject, among other references, an interesting general one is \cite[Main Theorem p.35]{Mi}, which was used in \cite[Lemma 32 \& Prop.7]{Fo--Kl4} in the context of {\it privileged primes} of the  ring  $\Z[i]$, the ring of Gaussian integers. The methods presented in \cite{Fo--Kl4} are easily translated in the context of $\Z [j]$ which is the theatre of our paper. We introduce the following notations: 

\noindent Let $a$ and $w$ be two  elements of $\Z[j]$ such that $w$ is coprime with $3a$. For $x\geq 2$, let 
$$
\pi_{\Z [j]} (x;w,a):= \bigl \vert
\{ \pi \in \Z [j]: \pi \text{ is a standard prime}, \, {\rm N} (\pi )\leq x, \, \pi \equiv a \bmod w\}
\bigr\vert,
$$
and let $\phi (w)$ be the number of invertible classes in $\Z [j]/ (w \Z [j])$. We then have 

\begin{proposition}
\label{SW*}
For every $A > 0$, there exists $c(A) >0$ such that, uniformly for
$$
x\geq 2,\, a,\, w \in \Z [j],\,  (w,3a) =1, \,  {\rm N}(w)\leq (\log x)^A, 
$$
one has the equality
$$
\pi_{\Z [j]} (x;w,a)=\frac{1}{\phi (w)} \pi_{\Z [j]} (x; 1, 0) + O\bigl(\, x \exp( -c(A)\, \sqrt{ \log x}\,)\, \bigr).
$$
\end{proposition}

This proposition gives the desired cancellation in sums over multiplicative characters $\chi$ on $\Z [j]$.

\begin{lemma} 
\label{sumchi(pi)**}
For every $A>0$, there exists $c(A) > 0$, such that, uniformly for $x \geq 2, w\in \Z [j]$, coprime with $3$ and satisfying $1 <{\rm N} (w) \leq  (\log x)^A$, $\chi$ a non-trivial character modulo $w$, one has the inequality 
$$
\sum_{\pi \textup{\, standard prime}\atop {\rm N} (\pi) \leq x} \chi (\pi)  = O \bigl( x \exp (- c(A) \sqrt {\log x}\, ) \bigr).
$$ 
In particular, for any $A>0$, there exists $C(A)$ such that, for any non-trivial character $\chi$ over $\Z[j]$, with period $w$, for every $x \geq 2$, one has the inequality
$$
\Bigl\vert \, \sum_{\pi \textup{\, standard prime}\atop {\rm N} (\pi) \leq x} \chi (\pi) \, \Bigr\vert  \leq   C(A)\,  x\, {\rm N}^{1/2}(w) \, (\log  x)^{-A}.
$$ 
\end{lemma}

\begin{proof}
We write the sum in question as
$$
\sum_{a \bmod w\atop (a, w) =1} \chi (a)\,  \pi_{\Z [j]} (x; w, a)
$$
and then apply  Proposition \ref{SW*}. Recalling that $\sum_{(a, w) = 1} \chi (a) = 0$ for a non-trivial character $\chi$ modulo $w$ finishes the proof.
\end{proof}

\subsubsection{Bounding $\vert \widetilde A (\boldsymbol \eta, \, \boldsymbol \zeta, \boldsymbol \epsilon, f_0, f_1) \vert$}  
\label{2.4.2}
Returning to the definitions \eqref{deftildeM*} and \eqref{concisedefinition*}
and applying Lemmas \ref{nontrivial**} and \ref{sumchi(pi)**}, we deduce that, for any $A > 0$, for any $f\in V^*$, for any $2 < U < Z$, for any $\boldsymbol \epsilon$ and $\boldsymbol e$ satisfying \eqref{812*}, we have
\begin{equation}
\label{932*}
\sum_{U <p_\infty <Z} \widetilde M (p_\infty) =
\sum_{\pi \text{ standard prime} \atop U < {\rm N} (\pi)<Z} M (\pi, f, \boldsymbol \epsilon, \boldsymbol e)= O_A \bigl( \Delta (f)^{1/2} Z (\log U)^{-A}\bigr).
\end{equation}
The constant implicit in the $O$--symbol depends on $A$ only. By the third and  fourth lines of \eqref{finalconditions*} we know that  $\Delta'_1$ is large, since it satisfies the inequality
$$
\Delta'_1\geq X^{1/4}\LL^{-A_0-A_2}.
$$
Furthermore, $p_\infty $ is the largest prime divisor of $\Delta'_1$ (see \eqref{Delta'1=*}) and $\Delta'_1$ has few prime factors (see the fifth line of \eqref{finalconditions*}) so we deduce the lower bound
$$
p_\infty \geq  (X^{1/4} \LL^{-A_0-A_2}\, \bigr) ^{1/A_3 \log \log X}\gg \exp \Bigl( \frac \LL{A_6 \log \log X}\Bigr),
$$
for some positive $A_6$. So we apply \eqref{932*} by choosing $U$ satisfying $\log U \gg \LL^{1/2}$ and $1 < \Delta (f) \leq \LL^{A_0}$ (see the third line of \eqref{finalconditions*}). The value of $Z$ is given by the last line of \eqref{finalconditions*} 
$$ 
Z=X^{1/4}\bigl/ \bigl( d^{3/2} \Delta_0^{3/2} \Delta_1^{3/2} \Delta''\bigr) . 
$$ 
Inserting these values in \eqref{932*}, we deduce, by \eqref{deftildeM*}, that
$$
\vert \widetilde A (\boldsymbol \eta, \, \boldsymbol \zeta, \boldsymbol \epsilon, f_0, f_1) \vert \ll_A X^{1/4}\, \bigl/\, \bigl(
d^{3/2}\, \Delta_0\,\Delta_1 \, {\Delta''_1} \, \LL^A\bigr)
$$
for any $A > 0$. Combining with  \eqref{A=A*}, with \eqref{LLA5*} and with \eqref{noteasy*}, we obtain the bound
\begin{multline*}
{\rm ET} (X, d) \ll X^{1/4} \LL^{A_5}\, \underset{\Delta_0,\, \Delta_1,\, \Delta''_1} {\sum\, \sum \, \sum}  \, 4^{\omega (\Delta_0)}\cdot 2^{\omega (\Delta_1)}\cdot 6^{\omega (\Delta''_1)}\,  \bigl(
d^{3/2}\, \Delta_0\,\Delta_1 \, {\Delta''_1} \, \LL^A\bigr)^{-1},
\end{multline*}
where $A$ is arbitrary. It remains to perform a crude summation  over $\Delta_0$, $\Delta_1$,  $\Delta''_1\, (< \Delta'_1)$ satisfying \eqref{finalconditions*} and  over $d \leq \LL^{A_4} $. By choosing $A$ sufficiently large we complete the proof of the following proposition

\begin{proposition}
\label{ETissmall*}
Uniformly for $X \geq 2$ one has
$$
\sum_{d \in \N_3^*\atop d\leq \LL^{A_4}} 2^{\omega (d)}\, {\rm ET}(X, d) = O ( X^{1/4}\LL^{-1}).
$$
\end{proposition}

\begin{remark} 
The orders of magnitude of the variables of summation $\Delta_0 \Delta_1$ and $\Delta'_1$  are completly different (see \eqref{finalconditions*}). So Lemma \ref{sumchi(pi)**} is the unique tool to exploit oscillation of characters. This situation is quite different from \cite{Fo--Kl2} or from \cite{Fo--Kl4}, for instance, where the case of variables with comparable sizes also has to be treated. This is accomplished by appealing to  bounds of {\it double oscillation type } (see \cite[Lemmas 14 \& 15] {Fo--Kl2}, \cite [\S 6]{Fo--Kl4} for instance).
\end{remark}

\subsection{\texorpdfstring{Study of ${\rm MT}(X, d)$}{Study of MT(X, d)}} 
\label{StudyofMT(X)*} 
We now turn our attention to the term ${\rm MT} (X, d)$, defined in \eqref{defMTdef*}. In order to prove that it behaves like a main term, we shall give the following asymptotic formula for the sum
$$
\sum_{d\leq \LL^{A_4}} 2^{\omega (d)} {\rm MT}(X,d)= C_{{\rm Heis^*}} X^{1/4} +O \bigl(\,X^{1/4}\LL^{-1}\,\bigr),
$$
(see \S \ref{finalstep}). By the definition \eqref{defMTdef*} we have
\begin{multline}
\label{975*}
{\rm MT} (X, d) = 2^{-2} 3^{-3}\  \underset {\Delta_0,\ \Delta_1,\ \Delta'_1}{ \sum \ \sum \ \sum }\quad    \underset{\substack{f _0,\, f'_0 \in V^* (\Delta_0)\\  f_1\in
V^* (\Delta_1)\\  f'_1 \in V^* (\Delta'_1)}}{ \sum\ \sum \ \sum \ \sum } \\
\prod_{r \mid \Delta'_1} \bigl\{ 1 +   \chi (f_0+f_1) (r) +\chi (2(f_0+f_1)) (r)    \bigr\}.
\end{multline}
The  factor $\prod_{r \mid \Delta'_1} \{\cdots\} $ is independent of the choice of $f'_0\in V^* (\Delta_0)$ and of  $f'_1 \in V^* (\Delta'_1)$. So we can replace  the summations over $f'_0$ and $f'_1$ by the factor $2^{\omega (\Delta_0)}\cdot 2^{\omega (\Delta'_1)}$. Furthermore the functions $f_0$ and $f_1$
only appear through  their sum $f:= f_0+f_1$. We rewrite $\Delta = \Delta_0  \Delta_1$. With this notation we have $f \in V^* (\Delta)$ 
and $\Delta (f) =\Delta$. Instead of summing over $f_0$, $f'_0$ and $f_1$, we sum over $f \in V^* (\Delta)$
and we  introduce  the factor
$$
\sum_{\Delta_0 \mid \Delta} 2^{\omega (\Delta_0)} = 3^{\omega (\Delta)}.
$$
Gathering these remarks, \eqref{975*} becomes
\begin{multline}
\label{987*}
{\rm MT} (X, d) = 2^{-2} 3^{-3}\  \underset {\Delta ,\ \Delta'_1}{   \sum \ \sum }\ 3^{\omega (\Delta )} \cdot 2^{\omega (\Delta'_1)}\, \\
\times \sum_{f \in V^* (\Delta)} 
\prod_{r \mid \Delta'_1} \bigl\{1 + \chi (f ) (r) + \chi (2f) (r)\bigr\} + O (X^{1/4} \LL^{-2}).
\end{multline}
The conditions of summation in \eqref{987*} are inferred from \eqref{finalconditions*}:
\begin{equation}
\label{995*}
\begin{cases}
d\Delta\, \Delta'_1 \in \N_{3}^*,\\
1 < \Delta \leq \LL^{A_0},\\
\Delta \, {\Delta'_1}^{2/3} \leq X^{1/6}/d.
\end{cases}
\end{equation}
The error term in \eqref{987*} comes from forgetting the fourth and the fifth lines of \eqref{finalconditions*}.  We control  the induced  error as it was done in  the proofs of Propositions \ref{smallDelta'} and \ref{largeomega}. The first condition of \eqref{995*} implies that $d$, $\Delta$ and $\Delta'_1$ are coprime in pairs.

\subsubsection{Expanding the product over primes $r$} 
By the multiplicativity of characters, the product appearing in \eqref{987*} equals
\begin{equation}\label{prod->sum}
\prod_{r \mid \Delta'_1} \bigl\{ 1 +   \chi (f ) (r) +\chi (2f) (r)    \bigr\} = \underset {d_0\, d_1 \, d_2 = \Delta'_1}{\sum \sum \sum} \,\chi (f) (d_1) \,\chi(2f) (d_2).
\end{equation}
We insert this expression in \eqref{987*} and we invert summations to obtain
\begin{multline}
\label{1008*}
{\rm MT} (X,d)= 2^{-2} 3^{-3}\  \sum_\Delta 
3^{\omega (\Delta )}  \sum_{f \in V^* (\Delta)} \underset{ d_1,\, d_2} {\sum\ \sum} 
\Bigl(2^{\omega (d_1)} \chi (f) (d_1)   \Bigr) 
\\
\times  \Bigl(   2^{\omega (d_2)} \chi (2f) (d_2)  \Bigr) \cdot \Bigl( \sum_{d_0}     2^{\omega (d_0)} \Bigr)
+O (X^{1/4} \LL^{-2}), 
\end{multline}
where the conditions of summation are deduced from \eqref{995*}
\begin{equation}
\label{1017*}
\begin{cases}
(d d_0 d_1 d_2)\,\Delta\,   \in \N_{3}^*,\\
1<\Delta \leq \LL^{A_0},\\
\Delta \, (d_0\, d_1\, d_2)^{2/3} \leq X^{1/6}/d.
\end{cases}
\end{equation}

\subsubsection{Controlling the sizes of $d_1$ and $d_2$} 
The last line of \eqref{1017*} implies that the product $d_1d_2$ can be as large as $X^{1/4}$. In that case \eqref{1017*} shows that the variable $d_0$ has no room for variation and Proposition \ref{complexanalysis*} below is useless in that situation (see formula  \eqref{perronapplied}). To circumvent this particular difficulty we invert summations as in the hyperbola method, to exploit the presence of the oscillating coefficients $\chi (f) (d_1)$ and $\chi(2f)(d_2)$. These non-trivial Dirichlet characters, with moduli $\ll \LL^{A_0}$, allow us to restrict the summation to 
$$
d_1,\,d_2 < D_0,
$$ 
where $D_0$ is a small power of $X$:
$$
D_0 := X^{1/100}.
$$
Indeed the contribution of the $(\Delta, d_0, d_1,d_2)$ to the right--hand side of \eqref{1008*} satisfying $\max (d_1,d_2)>D_0$ is negligible. To see this, consider for instance the case when $d_1 > D_0$. The corresponding contribution, denoted by $\Xi (D_0, d)$, is bounded by
\begin{equation}
\label{1034*}
\Xi (D_0,d) \ll \sum_\Delta 3^{\omega (\Delta)}\,\sum_{f\in V^* (\Delta)}
\sum_{d_0} 2^{\omega (d_0)}   \sum_{d_2}    2^{\omega (d_2)} \\
\Bigl\vert \, \sum_{d_1 } 2^{\omega (d_1)} \chi (f) (d_1)  \, 
\Bigr\vert
\end{equation}
where $D_0 <d_1\leq D_1:=X^{1/4}\Delta^{-3/2} d^{-3/2}d_0^{-1} \, d_2^{-1}.$

The Siegel--Walfisz Theorem allows us to save any power of $\LL$ over the trivial bound in the sum over $d_1$. More precisely, for any $A > 0$, one has the bound
\begin{equation}
\label{sumoverd1<*}
\sum_{D_0 < d_1 < D_1} 2^{\omega (d_1)} \chi (f) (d_1) \ll D_1 \, (\log D_0)^{-A}\ll D_1\, \LL^{-A}.
\end{equation}
Inserting this bound in \eqref{1034*}, summing over $\Delta$, $d_0$ and $d_2$, and choosing $A$ sufficiently large, we obtain the bound
\begin{equation}\label{Xi(D0)<<*}
\Xi (D_0,d) \ll X^{1/4} \LL^{-2}.
\end{equation}
We give some details about the proof of \eqref{sumoverd1<*}. The process is similar to what was explained in \S \ref{2.4.2}. First of all, one can restrict to $d_1$ with a reasonable number of prime factors, which means $\omega (d_1) \leq B_0 \log \log X$ for some $B_0$ with acceptable error by Lemma \ref{manyprimefactors}. The remaining $d_1$ are then factorized as $d_1 = p_\infty \delta_1$, where $p_\infty $ is the greatest prime factor of $d_1$. The prime $p_\infty$ is a large variable to which we can apply a Siegel--Walfisz Theorem related to the Dirichlet $L$--functions $L(s, \chi (f))$ and $L(s, \chi(f) \,( \frac {\cdot} 3))$. The second line of \eqref{1017*} ensures that the conductor of these $L$--functions is larger than $1$ but  less than $3 \LL^{A_0}$, which is the adequate situation to apply the Siegel--Walfisz Theorem. We omit the details.

In conclusion, by \eqref{Xi(D0)<<*}, we proved that \eqref{1008*} remains true, with the conditions of summations \eqref{1017*} replaced by 
\begin{equation}
\label{1055*}
\begin{cases}
(d d_0 d_1 d_2)\,\Delta\,   \in \N_{3}^*,\\
1<\Delta \leq \LL^{A_0},\\
d_1, \, d_2 \leq D_0,\\
\Delta \, (d_0\, d_1\, d_2)^{2/3} \leq X^{1/6}/d.
\end{cases}
\end{equation}

\subsubsection{Summing a multiplicative function on $\N_3^*$} 
To continue our study of the main term ${\rm MT} (X, d)$, as presented in \eqref{1008*}, we have to give a precise asymptotic expansion for $\sum_{d_0} 2^{\omega (d_0)}$. Actually we will study the following more general problem which is obviously linked with the possible extension of Theorem \ref{tMain} to any odd prime $\ell$: let $\ell \geq 3$ be prime, $d\geq 1$ an integer and $x\geq 1$ be a real number. We consider the sum 
\begin{equation*}
{\rm K} (x; \ell, d):= \sum_{\substack{ n\leq x, \, n\in \N_{\ell}^* \\ (n,d) =1} }\, (\ell -1)^{\omega (n)}.
\end{equation*}
Without loss of generality, we assume that
\begin{equation}
\label{condford*}
d\in \N_\ell^*.
\end{equation}
For the statement of our result, we denote by $\chi_0$, $\chi_1$,...,  $\chi_{\ell -2}$, the $\ell - 1$ Dirichlet characters modulo $\ell$, $\chi_0$ being the principal character. There is no risk of confusion with the notation introduced by \eqref{conventiononcharacter}. Let $\alpha_\ell$ be the infinite product
\begin{equation}
\label{Cell*}
\alpha_\ell := \frac \ell{\ell +1}   
\ \prod_p \Bigl\{\Bigl( 1 + \frac {1}{p}+\frac{\chi_1(p)}{p}+  \cdots + \frac{\chi_{\ell -2}(p)}{p}\Bigr)\cdot \Bigl( 1- \frac 1{p}\Bigr)
\Bigr\},
\end{equation}
and let $\psi_\ell (d)$ be the multiplicative function 
\begin{equation}
\label{defpsi*}
\psi_{\ell} (d) :=\prod_{p\mid d}\Bigl( 1 +\frac{(\ell -1)}p\Bigr)^{-1}.
\end{equation}
We will prove the following

\begin{proposition} 
\label{complexanalysis*}
Let $\ell \geq 3$ be a fixed prime. There exists $\nu = \nu_\ell > 0$ such that, uniformly for $d\geq 1$ satisfying \eqref{condford*} and $x\geq 2$, one has the equality 
$$
{\rm K} (x; \ell, d) = \alpha_{\ell} \, \psi_{\ell} (d)\,x + O \Bigl( \tau (d)^{\ell -1} x^{1-\nu}\Bigr).
$$
\end{proposition}
  
\begin{proof}   
Consider the  Dirichlet series
\begin{equation*} 
F(s)=F_{\ell, d} (s):= \sum_{ n\in \N_{\ell}^* \atop (n,d)=1}\,\frac{(\ell -1)^{\omega (n)}}{n^s}
= \sum_{n} \frac {a_n}{n^s},
\end{equation*}
by definition. This series is absolutely convergent in the half--plane $\{s : \sigma >1\}$. In this region, $F(s)$  has an expresion as an Euler product
$$
F(s)= \prod_{p\in \PP_\ell^* \atop p \nmid d} \Bigl(1 + \frac {(\ell -1)}{p^s}\Bigr).
$$
For a prime $p\not= \ell$ we detect the condition $p\equiv 1 \bmod \ell$, by the sum
$$
\frac 1{\ell -1}\bigl( \chi_0 (p) + \cdots + \chi_{\ell -2} (p) \bigr).
$$
Thus $F(s)$ has the following expression
\begin{align}
F(s)& = \prod_{p\nmid \ell d} \Bigl( 1 + \frac {\chi_0 (p)}{p^s}+ \cdots + \frac{\chi_{\ell -2}(p)}{p^s}\Bigr)\nonumber\\
& =\Bigl( 1 + \frac 1{\ell^s}\Bigr)^{-1} \ \prod_{p\mid d} \Bigl( 1+ \frac{(\ell -1)}{p^s}\Bigr)^{-1}\,  \prod_{p} \Bigl( 1 + \frac {1}{p^s}+\frac{\chi_1(p)}{p^s}+  \cdots + \frac{\chi_{\ell -2}(p)}{p^s}\Bigr).\label{148*}
\end{align}
Recall the following Euler products for $\sigma  >1$:
$$
\zeta (s)^{-1} = \prod_p \Bigl(1 -\frac 1 {p^s}\Bigr),
$$
and
$$
L(s, \chi_j) ^{-1} = \prod_{p} \Bigl( 1 - \frac{\chi_j (p)}{p^s}\Bigr)\ (1\leq j \leq \ell-2).
$$
Inserting these products into \eqref{148*}, we have the equality
\begin{equation}
\label{159*}
F(s) =  \zeta (s) \Bigl[ \Big( 1+ \frac 1{\ell^s}\Bigr)^{-1} \cdot \prod_{p\mid d} \Bigl( 1 +\frac{(\ell -1)}{p^s}\Bigr)^{-1}\cdot  L(s, \chi_1)\cdots L (s, \chi_{\ell-2})\Bigr] G(s),
\end{equation}
where $G(s)$ is defined by the Euler product
\begin{equation*}
G(s):= \prod_p \Bigl\{\Bigl( 1 + \frac {1}{p^s}+\frac{\chi_1(p)}{p^s}+  \cdots + \frac{\chi_{\ell -2}(p)}{p^s}\Bigr)\cdot \Bigl( 1- \frac 1{p^s}\Bigr)\cdot 
\Bigl( 1- \frac {\chi_1 (p)}{p^s}\Bigr) \cdots \Bigl( 1- \frac {\chi_{\ell-2} (p)}{p^s}\Bigr)
\Bigr\}.
\end{equation*}
Actually this Euler product is absolutely convergent in the half--plane $\{ s : \sigma >1/2\}$. Returning to \eqref{159*} we proved that the Dirichlet series $F(s)$  has a meromorphic continuation of the form
$$
F(s) = \zeta (s) H(s), 
$$
where $H(s)= H_{\ell, d} (s)$ is holomorphic on the half plane
$$
\Omega := \bigl\{ s : \sigma >(\log (\ell -1))/\log (\ell +1)\bigr\}.
$$ 
On this half--plane, $F(s)$ has a unique pole at $s=1$. This pole is induced by the singularity of $\zeta$ at $s = 1$. Hence this pole of $F$ is simple with residue
\begin{multline*}
{\rm Res} (F; s=1) = H_{\ell, d}(1) = \psi_{\ell} (d)\cdot 
\frac \ell{\ell +1}   \cdot \Bigl[  L(1, \chi_1)\cdots L (1, \chi_{\ell-2})\Bigr] \\
\times \prod_p \Bigl\{\Bigl( 1 + \frac {1}{p}+\frac{\chi_1(p)}{p}+  \cdots + \frac{\chi_{\ell -2}(p)}{p}\Bigr)\cdot \Bigl( 1- \frac 1{p}\Bigr)
\Bigl( 1- \frac {\chi_1 (p)}{p}\Bigr) \cdots \Bigl( 1- \frac {\chi_{\ell-2} (p)}{p}\Bigr)\Bigr\},
\end{multline*}
which equals
$$
{\rm Res} (F; s=1) = \alpha_\ell\, \psi_\ell (d).
$$
The  number $\alpha_\ell$  is not zero as a  consequence of the fact that $L(1, \chi_j) \neq 0$. We apply an effective version of Perron's formula (see for instance \cite[Corollary 5.3, p. 140]{MoVa}) to obtain the equality
\begin{multline}
\label{K(x;ell,d)=*}
{\rm K} (x; \ell, d) =\frac 1{2\pi i} \int_{\kappa -iT}^{\kappa +iT}  F(s) \frac {x^s}s {\rm d} s \\+ O \Bigl( \sum_{x/2 < n <2x\atop n\not= x} 
\vert a_n \vert \min \Bigl( 1, \frac{x}{T\vert x-n\vert}\Bigr)\Bigr) +O \Bigl( \frac{ 4^\kappa +x^\kappa}T\ \sum_{n=1}^\infty \frac {\vert a_n\vert}{n^{\kappa}}\Bigr)+ O (x^\varepsilon).
\end{multline}
If we choose $\kappa = 1 + 2\varepsilon$, and $T = x^{\vartheta}$ ($\vartheta > 0$), we have the equality 
$$
{\rm K} (x; \ell, d) =\frac 1{2\pi i} \int_{\kappa -iT}^{\kappa +iT}  F(s) \frac {x^s}s {\rm d} s + O(x^{1-\vartheta +\varepsilon})
$$
by the inequality $\vert a_n \vert \ll n^\varepsilon$ and by separating the cases  $\vert x-n\vert <x/T$ and $\vert x -n \vert \geq  x/T$ in the first sum  on the right--hand side of \eqref{K(x;ell,d)=*}. 
 
We transform the path of integration into a vertical segment  $\sigma_0 +it$ with $\sigma_0 <1$ and  $\vert t \vert \leq T$  belonging to $\Omega$ and two horizontal segments belonging to the lines with equations $t = T$ and $t = -T$.  On these segments, the function $G(s)$, defined in \eqref{159*}, is uniformly bounded and we also have 
$$
\Big( 1+ \frac 1{\ell^s}\Bigr)^{-1} \cdot \prod_{p\mid d} \Bigl( 1 +\frac{(\ell -1)}{p^s}\Bigr)^{-1}=O \bigl(\tau (d)^{\ell -1}\bigr).
$$
By classical bounds for the functions $L(s, \chi_j)$ on these segments, by an optimal choice of $\vartheta$ and $\sigma_0$,  
we complete the proof of Proposition \ref{complexanalysis*}. 
\end{proof}

We apply  Proposition \ref{complexanalysis*} with the values 
$$
n \leftarrow d_0, \ \ell \leftarrow 3, \ d \leftarrow dd_1d_2\Delta, \ x\leftarrow X^{1/4} d^{-3/2} d_1^{-1} d_2^{-1}\Delta^{-3/2}
$$
to obtain the equality
\begin{equation}
\label{perronapplied}
\sum_{d_0} 2^{\omega (d_0)} = \alpha_3\, \psi_3 \bigl(dd_1d_2 \Delta\bigr)\,  \frac {X^{1/4}}{d^{3/2} d_1 d_2 \Delta^{3/2}} + O \Bigl( \tau^2 (dd_1d_2 \Delta)
\, \Bigl( \frac {X^{1/4}}{d^{3/2} d_1 d_2 \Delta^{3/2}} \, \Bigr)^{1-\nu}\Bigr).  
\end{equation}
Denote by $\mathcal Er (X, d, d_1,d_2, \Delta)$ the error term in the above formula. Since we have the inequalities $d \leq \LL^{A_4}$ (see \eqref{sumoverd}), $d_1, \, d_2 \leq D_0$ and $\Delta \leq \LL^{A_0}$ (see \eqref{1055*}, we see that the total contribution to ${\rm Heis}^\dag (X)$ will be negligible, since we have (see \eqref{sumoverd} and \eqref{U=MT+ET})
\begin{equation}
\label{1228*}
\sum_d \, \sum_{d_1}\sum_{d_2} \, \sum_{\Delta} 2^{\omega (d)} \mathcal Er (X, d, d_1,d_2, \Delta) = O\bigl( X^{1/4-\delta}\bigr),
\end{equation}
for some positive $\delta$. This contribution is compatible with the error term that we claim in \eqref{asympforHeisdag}.

\subsubsection{The final step}
\label{finalstep} 
We insert the equality \eqref{perronapplied} in \eqref{1008*}. By  \eqref{sumoverd}, \eqref{U=MT+ET}, \eqref{1034*}, \eqref{Xi(D0)<<*}, \eqref{1228*} and Proposition \ref{ETissmall*}, we see that, in order to prove \eqref{asympforHeisdag}, it is sufficient to prove that the sum ${\rm Heis}^\ddag (X)$ defined by 
\begin{multline}
\label{1235*}
{\rm Heis}^\ddag (X):= 2^{-2} 3^{-3}\alpha_3\sum_{d} \psi_3 (d)\cdot \frac {2^{\omega (d)}}{d^{3/2}}
\sum_\Delta \psi_3(\Delta)\cdot 
\frac{3^{\omega (\Delta )}}{\Delta^{3/2}} \\ \sum_{f \in V^* (\Delta)} \Bigl( \sum_{d_1} 
\psi_3 (d_1) 2^{\omega (d_1)} \frac{\chi (f) (d_1) }{d_1}  \Bigr) 
\Bigl( \sum_{d_2} \psi_3 (d_2)  2^{\omega (d_2)} \frac{ \chi (2f) (d_2) }{d_2} \Bigr),
\end{multline}
with the conditions of summations
\begin{equation}\label{1243*}
\begin{cases}
(d d_1 d_2)\,\Delta\,   \in \N_{3}^*,\\
1<\Delta \leq \LL^{A_0},\\
d\leq \LL^{A_4},\\
d_1, \, d_2 \leq D_0.\\
\end{cases}
\end{equation}
satisfies the equality
\begin{equation}\label{1252*}
{\rm Heis}^\ddag (X) =C_{{\rm Heis}^*}  +O (\LL^{-1}).
\end{equation}
Once again by the Siegel--Walfisz Theorem, we can drop the conditions $d_1, d_2 \leq D_0$ in \eqref{1243*} with an error in $O(\LL^{-1})$ so that complete series over $d_1$ and $d_2$ appear. By the equality $\chi (f) (d) =0$  if $(d, \Delta (f)) >1$ and the value $\psi_3 (p) =p/(p+2)$ we have the equality 
\begin{align*}
\Bigl( \sum_{d_1}\Bigr) \Bigl( \sum_{d_2}\Bigr) &=
\sum_{d_1\in \N_3^* \atop (d_1,d) =1}
\psi_3 (d_1) 2^{\omega (d_1)} \frac{\chi (f) (d_1) }{d_1}  \prod_{p\in \PP_3^* \atop p\nmid dd_1} \Bigl( 1+ 2 \frac {\chi (2f)(p)}{p+2}\, \Bigr)\\
&= \Bigl\{\, \prod_{p \in \PP_3^*\atop p \nmid d} \Bigl( 1 + 2\frac {\chi (2f)(p)}{p+2}\Bigr) \,\Bigr\}\cdot \Bigl\{\,  \sum_{d_1\in \N_3^* \atop (d_1,d) =1}
\frac{\psi_3 (d_1) 2^{\omega (d_1)} \chi (f) (d_1)}{d_1 \prod_{p\in \PP_3^* \atop p\mid d_1} \Bigl( 1+ 2 \frac {\chi (2f)(p)}{p+2}
\Bigr)} \,\Bigr\}\\
&= \Bigl\{\prod_{p \in \PP_3^*\atop p \nmid d} \Bigl( 1 + 2\frac {\chi (2f)(p)}{p+2}\Bigr)\, \Bigr\}
\cdot \Bigl\{ \prod_{p \in \PP_3^*\atop p\nmid d} \Bigl( 1 + \frac{2\chi (f) (p)}{p+2+2 \chi (2f) (p)} \Bigr)\, \Bigr\}\\
& =  \prod_{p \mid d} \Bigl( \, 1 +2 \frac{ \chi (f) (p) +\chi (2f) (p)}{p+2}\, \Bigr)^{-1} \prod_{p\in \PP_3^*} \Bigl( \, 1 +2 \frac{ \chi (f) (p) +\chi (2f) (p)}{p+2}\, \Bigr).
\end{align*}
We insert this value in \eqref{1235*}, and invert the summations. We extend the summation to all $d\in \N_3^*$ and all $\Delta \in \N_3^*$ with $ \Delta >1$ and $(d, \Delta) = 1$. With an acceptable error in $O (\LL^{-1})$, we have the equality
\begin{multline*}
{\rm Heis}^\ddag (X) = 2^{-2} 3^{-3} \alpha_3 \sum_{\Delta \in \N_3^*\atop \Delta >1} \psi_3 (\Delta)\cdot \frac{3^{\omega (\Delta)}}{\Delta^{3/2}} 
\sum_{f\in V^* (\Delta)} \Bigr\{\, \prod_{p\in \PP_3^*} \, \Bigl( \, 1 +2 \frac{ \chi (f) (p) +\chi (2f) (p)}{p+2}\, \Bigr)\,\Bigr\}\\
\times 
\Bigl\{ \, \prod_{p \in \PP_3^*\atop p \nmid \Delta} \Big( 1 + \frac 2{p^{1/2}\bigl(\, p+2 (1+\chi (f) (p) +\chi (2f) (p))\, \bigr)} \Bigr)\, 
\Bigr\}
+ O (\LL^{-1}).
\end{multline*}
We recognize the constant $C_{{\rm Heis}^*}$ defined in \eqref{defc*}. So we proved \eqref{1252*} and the proof of Proposition \ref{archetype} is now complete.

\subsection{\texorpdfstring{Comments on the constant $C_{{\rm Heis}^*}$}{Comments on the leading constant}}
\label{ssComments}
We will prove the following

\begin{proposition}
\label{1289*}
The constant $C_{\rm Heis^*}$ is a real positive number.
\end{proposition} 

\begin{proof}
It follows from definition \eqref{defc*} that $C_{\rm Heis^*}$ is a real non-negative number, since it is a sum of non-negative real numbers. To prove that $C_{\rm Heis^*} > 0$, it is sufficient to prove that for at least one $\Delta \in \N_3^*, \, \Delta >1$ and one $f \in V^* (\Delta)$, we have
\begin{multline*}
\Bigr\{\, \prod_{p\in \PP_3^*} \, \Bigl( \, 1 +2 \frac{ \chi (f) (p) +\chi (2f) (p)}{p+2}\, \Bigr)\,\Bigr\}\\
\times 
\Bigl\{ \, \prod_{p \in \PP_3^*\atop p \nmid \Delta} \Big( 1 + \frac 2{p^{1/2}\bigl(\, p+2 (1+\chi (f) (p) +\chi (2f) (p))\, \bigr)} \Bigr)\, 
\Bigr\} > 0.
\end{multline*}
By the inequality $1+\chi (f) (p) +\chi (2f) (p) \geq 0$, the second product is  an absolutely convergent product, the limit of which is positive. We will prove the following lemma which implies Proposition \ref{1289*}

\begin{lemma}
\label{1305*} 
We have for every $\Delta \in \N_3^*$ with $\Delta > 1$ and for every $f\in V^* (\Delta)$
$$
\prod_{p\in \PP_3^*} \, \Bigl( \, 1 +2 \frac{ \chi (f) (p) +\chi (2f) (p)}{p+2}\, \Bigr) > 0.
$$
\end{lemma}

To prove this lemma, we will approximate this infinite product, that we denote by $\mathcal P (f)$, by a product of the values at the point $s=1$ of four Dirichlet $L$--series attached to  characters of orders $3$ or $6$. Each factor of $\mathcal P (f)$ is a positive real number. If $p \neq 3$, we detect the congruence $p \equiv 1 \bmod 3$ by the sum $(1+ (p/3))/2$. We have
\begin{align}
\label{ePf}
\mathcal P (f)& = \prod_{p\not= 3} \Bigl( \, 1 + \bigl(1+(\frac p3)\,\bigr) \cdot \frac {\chi (f) (p) + \chi (2f)(p)}{p+2}\,\Bigr) \nonumber \\ 
& = \prod_{p\not= 3} \Bigl( 1+ \frac {\chi (f) (p)}p\Bigr)  \overline{\Bigl( 1+ \frac {\chi (f) (p)}p\Bigr) } \nonumber \\
&\qquad \qquad \qquad \times \Bigl( 1+ \frac { (p/3)\chi (p)}p\Bigr)  \overline{\Bigl( 1+ \frac {(p/3)\chi (f) (p)}p\Bigr)} \Bigl( 1+ \frac {\xi (p)}{p^2}\Bigr), 
\end{align}
where $\xi (p)$ is  some unspecified  real number satisfying  $1+ \xi (p)/p^2>0$ and $\xi (p) =O(1).$ We introduce the factor corresponding to the prime $p=3$ and we continue the transformations of $\mathcal P (f)$ to arrive at the equality
$$
\mathcal P (f) = \bigl\vert L\bigl(\,1, \chi (f)\,\bigr)\vert^2 \cdot  \bigl\vert L\bigl(\, 1, (\cdot/3) \chi (f)\,\bigr)\bigr\vert^2 \prod_{p\geq 2} 
\Bigl( 1+ \frac {\xi' (p)}{p^2}\Bigr), 
$$
where $\xi' (p)$ is  another unspecified  real number satisfying  $1+ \xi' (p)/p^2>0$ and $\xi' (p) =O(1).$ The inequalities $ \bigl\vert L\bigl(\,1, \chi (f)\,\bigr)\vert^2 >0$,  $\bigl\vert L\bigl(\, 1, (\cdot/3) \chi (f)\,\bigr)\bigr\vert^2>0$ and  $\prod_{p\geq 2} 
\Bigl( 1+ \frac {\xi' (p)}{p^2}\Bigr) >0$ imply $\mathcal P (f) >0$.  This gives Lemma \ref{1305*} and also Proposition \ref{1289*}.
\end{proof}

\section{Study of the other sums} 
\label{theothersums}
We now study the thirteen sums  ${\rm Heis}^{(i, j)} (X)$ for $(i, j)\not= \eqref{C1}$ by comparison with ${\rm Heis}^{\eqref{C1}}(X)={\rm Heis}^* (X)$, the asymptotic value  of which is given in Proposition \ref{archetype}.

\subsection{\texorpdfstring{Easy observations between pairs of ${\rm Heis}^{(i, j)}(X)$}{Easy observations between pairs of Heis(i, j)(X)}} 
By inspecting the list of conditions \eqref{C1},\dots, \eqref{C14}, we see that we pass from \eqref{C1} to \eqref{C8}, from \eqref{C2} to \eqref{C9},\dots, from \eqref{C7} to \eqref{C14}, by replacing the condition $3\nmid d$ by $3 \mid d$. By studying Definition \ref{defmu(d)} and definition \eqref{defD}, we easily get

\begin{lemma} 
Let $d$ be an element of $\N_{3}^*$ and let $f, f' \in V$. Then we have the equality 
$$
D(3d, f, f')
=
\begin{cases}
3^{ {12}} \cdot D(d,f, f') &\textup{ if } f(3)=f'(3)=0,\\
D(d,f, f' ) & \textup{ otherwise.}
\end{cases}
$$
\end{lemma}

We now follow the influence of the conditions $3 \nmid d$ and $3\mid d$ in the value of the sum $S(X,f, f')$ defined in \eqref{definitionS3} (recall that  $\Delta (f) \Delta (f')$ is coprime with $3$ and that ${\rm free }(3d, 3) =d$ for $d\in \N_{3}^*$). This gives the following

\begin{proposition}
\label{1494} 
We have the equalities
$$
{\rm Heis}^{\eqref{C1}} (3^{ {-12}}X) =   {\rm Heis}^{\eqref{C8}} (X),$$
and   
$$ {\rm Heis}^{\eqref{C2}} (X) =   {\rm Heis}^{\eqref{C9}} (X),  \, {\rm Heis}^{\eqref{C3}} (X) =  {\rm Heis}^{\eqref{C10}} (X),$$
$$ {\rm Heis}^{\eqref{C4}} (X) =   {\rm Heis}^{\eqref{C11}} (X), \, {\rm Heis}^{\eqref{C5}} (X) =   {\rm Heis}^{\eqref{C12}} (X),$$
 $$ {\rm Heis}^{\eqref{C6}} (X) =   {\rm Heis}^{\eqref{C13}} (X),\, { \rm Heis}^{\eqref{C7}} (X) =   {\rm Heis}^{\eqref{C14}} (X).$$
\end{proposition}

\noindent The first part of this proposition, combined with Proposition \ref{archetype}, shows that
$$
C^{\eqref{C8}} = 3^{ {-3}} H_0.
$$ 
Moreover the second part of  Proposition \ref{1494} reduces the proof of Proposition \ref{allthecases} to the study of six sums: ${\rm Heis}^{\eqref{C2}}(X) $,  ${\rm Heis}^{\eqref{C3}}(X)$, ${\rm Heis}^{\eqref{C4}}(X)$, ${\rm Heis}^{\eqref{C5}}(X)$,  ${\rm Heis}^{\eqref{C6}}(X)$ and ${\rm Heis}^{\eqref{C7}}(X)$.

\subsection{\texorpdfstring{Preparation of the functions $f$ and $f'$}{Preparation of the functions f and f'}}
In the six remaining sums, we remark that the prime $3$ belongs to ${\rm supp \, } f \cup {\rm supp \, }f'$. We generalize the decomposition \eqref{decompf*} as follows

\begin{equation}
\label{doubledecomposition}
\begin{cases}
f&= \eta \,\mathbbm 1_{\{3\}} \oplus f_0 \oplus f_1,\\
f'&= \eta'\, \mathbbm 1_{\{3\}} \oplus f'_0 \oplus f'_1,
\end{cases} 
\end{equation}

\begin{itemize}
\item where $\eta, \eta' \in \{0, 1, 2\}$, 
\item where  $\mathbbm{1}_{\{3\}}$ is defined in \S \ref{crucialsum}, 
\item where the functions $f_0$, $f'_0$, $f_1$ and $f'_1$ do not contain $3$ in their support, 
\item where we have ${\rm supp\, } f_0 = {\rm supp \, } f'_0 \ (:= \mathcal E_0)$, 
\item where the three sets $\mathcal E_1$ ($ := {\rm supp \,} f_1$),
$\mathcal E'_1$ ($:= {\rm supp\,} f'_1$) and $\mathcal E_0$ are disjoint.
\end{itemize}

\noindent This decomposition is unique and the definitions of $\Delta_0$, $\Delta_1$ and $\Delta'_1$ (see \eqref{decompDelta*}) remain valid. Observe that $\Delta_0\Delta_1\Delta'_1$ is never divisible by $3$. We now state a generalization of Lemma \ref{expand1}, which can be proven in the same way as Lemma \ref{expand1}.

\begin{lemma}
\label{expand2} 
Let $f, f'\in V$ decomposed as in \eqref{doubledecomposition}. We then have the equalities  
\begin{multline*}
\underset{(z,z')\in \F_3^2 \atop  f(r) z +f'(r) z' =0}{\sum \sum }
\bigl( \chi (zf+z'f') \bigr) (r) = \\
1+
\begin{cases} 
\chi (f') (r) +\chi (2f') (r) &\textup{ if } r\in \mathcal E_1,\\
\chi (f) (r) +\chi (2f) (r) &\textup{ if } r\in \mathcal E'_1,\\
\chi \bigl(f'_0 (r) f+ 2f_0(r)f' \bigr)(r) 
+ \chi \bigl(2f'_0 (r) f+ f_0 (r)f'\,\bigr) (r) &\textup{ if } r \in \mathcal E_0.\\
\end{cases}
\end{multline*}
\end{lemma}

As a consequence of this lemma, we deduce that the triple product appearing at the end of \eqref{sumsumsumsumsumsumsum*} now has the shape
\begin{multline}\label{tripleprod}
\Pi (f,f') := \prod_{r\mid \Delta_0}\Bigl\{ 1+ \chi \bigl(f'_0 (r) f+ 2f_0(r)f' \bigr)(r) 
+ \chi \bigl(2f'_0 (r) f + f_0 (r)f'\,\bigr) (r)\Bigr\}\\
\times \prod_{r\mid \Delta_1}\bigl\{ 1+\chi (f') (r) +\chi (2f') (r) \bigr\}
\prod_{r \mid \Delta'_1} \bigl\{ 1 +   \chi (f) (r) +\chi (2f) (r)    \bigr\}.
\end{multline}
As in \S \ref{decompU(X,d)}, we write this product in a schematic way as
$$
\Pi (f,f') = \prod_{r \mid \Delta_0}\{ \cdots\}\prod_{r\mid \Delta_1} \{ \cdots\} \prod_{r \mid \Delta'_1} \{ \cdots \}.
$$
In the six sums, that we will study below, the main term will correspond  to the contribution of the subproduct $\Pi^{\rm{mt}}(f,f')$ of $\Pi (f,f')$ defined by 
\begin{equation}
\label{Pimt}
\Pi^{\rm{mt}}(f,f'):=\prod_{r \mid \Delta'_1} \{ \cdots\}, 
\end{equation}
while the complementary product $\Pi^{\rm {et}} (f,f')$, defined by
\begin{equation*}
\Pi^{\rm{et}}(f,f') : = \Bigl( -1 + \prod_{r \mid \Delta_0} \{ \cdots \} \prod_{r\mid \Delta_1}\{ \cdots \} \Bigr) \, \prod_{r\mid \Delta'_1} \{ \cdots \},
\end{equation*}
is absorbed in the error term after summation over $d$, $\Delta_0$, $\Delta_1$, $\Delta'_1$, $f$, $f'$. 

\subsection{\texorpdfstring{Study of ${\rm Heis}^{\eqref{C2}}(X)$}{Study of Heis(X), I}}
\label{sectionC2}
In this case we have $\mu(f, f',d) = 3^8$ which incites to compare ${\rm Heis}^{\eqref{C2}} ( X)$ with ${\rm Heis}^* (X/3^8)$. By \eqref{C2}, we need to impose three conditions on the functions $f$ and $f'$ that we decompose as in \eqref{doubledecomposition}. The first condition is $f(3)= \eta = 0$ and is equivalent to $f \in V^*$. The second condition $f'(3) \neq 0$ (i.e. $\eta' =1$ or $2$) does not affect the treatment of the error terms $\Pi^{{\rm et}} (f,f')$. More precisely, we separate the cases $\eta'=1$ and $\eta'=2$. Then we follow the technique used in \S \ref{StudyofET},  which benefits, after some preparation, from the oscillation of a non principal Dirichlet character (with modulus less than some fixed power of $\LL$). Then we obtain an analogue of Proposition \ref {ETissmall*}. 

To deal with the contribution of the main term $\Pi^{\rm mt}(f,f')$ defined in \eqref{Pimt}, we use the decomposition \eqref{doubledecomposition} of $f'$. This means that in \eqref{975*}, we have to introduce an extra summation over $\eta' \in \{1, 2\}$. Gathering these remarks, taking care of the third condition $\chi (f) (3)=1$ in \eqref{C2} and appealing to the definition \eqref{defH1} of $H_1$, we conclude that

\begin{proposition}
\label{PropforC2}
Uniformly for $X\geq 2$, one has the equality
\begin{equation*}
{\rm Heis}^{\eqref{C2}} (X) = 2^{-1} \cdot 3^{-5} \alpha_3    \, H_1\,  X^{1/4}+O (X^{1/4} \LL^{-1}).
\end{equation*}
\end{proposition}

\subsection{\texorpdfstring{Study of ${\rm Heis}^{\eqref{C3}} (X)$}{Study of Heis(X), II}}  
We now have $\mu (f, f',d) = 3^{12}$,  which incites to compare ${\rm Heis}^{\eqref{C3}} (X)$ with ${\rm Heis}^* (X/3^{12})$. Furthermore, as in \S \ref{sectionC2} we have $\eta = 0$ and $\eta' \in \{1, 2\}$. Following the proof of Proposition \ref{PropforC2}, we get
\begin{equation*}
{\rm Heis}^{\eqref{C3}} (X) = 2^{-1} \cdot 3^{-6} \alpha_3    \, H'_1\,  X^{1/4}+O (X^{1/4} \LL^{-1})
\end{equation*}
with 
\begin{multline*} 
H'_1 := \sum_{\Delta \in \N_3^*\atop \Delta >1} 
\lambda  (\Delta)\,  \psi_3 (\Delta)\cdot \frac{  {3}^{\omega (\Delta)}}{\Delta^{3/2}} 
\sum_{f\in V^* (\Delta)\atop \chi (f) (3) =j,\, j^2} \\
\Bigr\{\, \prod_{p\in \PP_3^*} \, \Bigl( \, 1 +2 \frac{ \chi (f) (p) +\chi (2f) (p)}{p+2}+\frac 2{p^{1/2}(p+2)} \Bigr)\,\Bigr\}.
\end{multline*}
Applying \eqref{chi(f)=0or} and returning to the definitions of $H_0$ and $H_1$ (see \eqref{defH0} and \eqref{defH1}), we trivially have the equality
$$
H_1 + H'_1 = H_0.
$$
So we proved the following

\begin{proposition}
\label{PropforC3}
Uniformly for $X\geq 2$, one has the equality
\begin{equation*}
{\rm Heis}^{\eqref{C3}} (X) = 2^{-1} \cdot 3^{-6} \alpha_3 \, (H_0 -H_1)\,  X^{1/4} + O (X^{1/4} \LL^{-1}).
\end{equation*}
\end{proposition}

\subsection{\texorpdfstring{Study of ${\rm Heis}^{\eqref{C4}} (X)$}{Study of Heis(X), III}} 
\label{2116}
In this case we have 
\begin{equation}
\label{mu=312}
\mu (f, f',d) = 3^{12}.
\end{equation}
By the conditions \eqref{C4}, we know that in  the decomposition \eqref{doubledecomposition}, we have $\eta \in \{1, 2\}$ and $\eta' = 0$. Furthermore  the functions $f$ and $f'$ are linearly independent if and only if $\Delta (f) \geq 1$ and $\Delta (f') >1$.  Since $\chi (f') (3) \not=0$  (see \eqref{chi(f)=0or})  we detect the condition $\chi (f')(3) = 1$ by the sum
$$
\frac 13 \Bigl( 1 + \chi (f')(3)+ \chi (2f') (3)\Bigr),
$$
and this factor is easily integrated in the second product on the right--hand side of \eqref{tripleprod} by replacing the product over $r \mid \Delta_1$ by $r \mid 3 \Delta_1$. This extra factor causes no new difficulty in the treatment of the error term: one follows the method explained in \S \ref{sectionC2}. 

The treatment of the main term requires more care. Up to some error in $O( X^{1/4} \LL^{-1})$ the main term has the shape (compare with \eqref{975*})
$$
2^{-2} 3^{-4} \sum_{d\in \N_3^*\atop d\leq \LL^{A_4}}2^{\omega (d)}   \sum_{(\eta, \eta')  \in \{(1,0), (2,0)\}}  \underset {\Delta_0,\ \Delta_1,\ \Delta'_1}{ \sum \ \sum \ \sum }\quad    \underset{\substack{f _0,\, f'_0 \in V^* (\Delta_0)\\  f_1\in
V^* (\Delta_1)\\  f'_1 \in V^* (\Delta'_1)}}{ \sum\ \sum \ \sum \ \sum }\, 
\Pi^{\rm mt} (f,f'),
$$
where
\begin{itemize}
\item we use the notations of \eqref{doubledecomposition},
\item the conditions  of summations are given by \eqref{finalconditions*}, but with $X$ replaced by $X/3^{12}$ (consequence of \eqref{mu=312}).
\end{itemize}
When we expand the product over $r\mid \Delta'_1$ appearing in the definition \eqref{Pimt} we have the following analogue of \eqref{prod->sum}
$$
\Pi^{\rm mt}(f,f')= 
\prod_{r\mid \Delta'_1} \bigl\{ \cdots    \bigr\}= \underset{d_0d_1d_2 =\Delta'_1}{ \sum\sum \sum}\, \chi (f_0+f_1 +\eta \mathbbm 1_{\{3\}}) (d_1)
\chi (2(f_0+f_1 +\eta \mathbbm 1_{\{3\}}) )(d_2)
$$
(we recall that  $\eta \in \{1, 2\}$). We now write $f = f_0 + f_1$ to mimic the notations used in \S \ref{StudyofMT(X)*} and we follow the method given in that section. By the definition \eqref{defH2}, we finally arrive at 

\begin{proposition}
\label{PropforC4}
Uniformly for $X\geq 2$, one has the equality
\begin{equation*}
{\rm Heis}^{\eqref{C4}} (X) = 2^{-2} 3^{-7} \alpha_3 \,    H_2\,  X^{1/4}+O (X^{1/4} \LL^{-1}).
\end{equation*}
\end{proposition}

\subsection{\texorpdfstring{Study of ${\rm Heis}^{\eqref{C5}}(X)$}{Study of Heis(X), IV}}
We now have 
\begin{equation}
\label{mu=316}
\mu (f, f',d) = 3^{16},
\end{equation}
$\eta \in \{1, 2\}$ and $\eta' = 0$. By \eqref{chi(f)=0or}, the event $\chi (f')(3) \in \{j, j^2\}$  is complementary to the event $\chi (f') (3) = 1$ treated in \S \ref{2116}. We detect the condition $\chi (f') (3) = j$ and the condition $\chi (f') = j^2$, by the respective indicators
\begin{equation}\label{twopossibilities}
\frac 13 \Bigl( 1 + j^2 \, \chi (f') (3) + j \, \chi (f') (3)\Bigr) \text{ and } \frac 13 \Bigl( 1 + j \, \chi (f') (3) + j ^2\, \chi (f') (3)\Bigr),
\end{equation}
which can also be incorporated in the right--hand side of \eqref{tripleprod} by  replacing the product over $r \mid \Delta_1$ by $r \mid 3\Delta_1$.   We  now follow the proof of Proposition \ref{PropforC4}. By taking into account the value of $\mu (f,f',d)$ given in \eqref{mu=316} and the two cases listed in \eqref{twopossibilities}, we complete the proof of 

\begin{proposition}
\label{PropforC5}
Uniformly for $X\geq 2$, one has the equality
\begin{equation*}
{\rm Heis}^{\eqref{C5}} (X) = 2^{-1} \cdot 3^{-8} \alpha_3\, H_2\,  X^{1/4}+O (X^{1/4} \LL^{-1}).
\end{equation*}
\end{proposition}

\subsection{\texorpdfstring{Study of ${\rm Heis}^{\eqref{C6}}(X)$}{Study of Heis(X), V}} 
We now have
\begin{equation}
\label{mu=312again}
\mu (f, f',d) = 3^{12}
\end{equation}
and $\eta, \eta' \in \{1, 2\}$. This condition implies that 
$$
\chi (f' (3)\cdot f + 2 f(3) \cdot f')(3) \neq 0 
$$ 
by \eqref{convention}. We detect the equality $\chi (f' (3)\cdot f + 2 f(3) \cdot f')(3) = 1$ by the sum
$$
\frac 13 \Bigl( 1 +  \chi (f' (3)\cdot f + 
2 f(3) \cdot f')(3) +  \chi (2f' (3)\cdot f + 
f(3) \cdot f')(3)\Bigr),
$$
which is easily inserted in the first product on the right--hand side of \eqref{tripleprod} by changing the product $\prod_{r\mid \Delta_0}$ to $\prod_{r \mid 3\Delta_0}$. The treatment of the error term is the same as for the archetype sum. For the main term we take into account the four values $(\eta, \eta') \in \{ 1, 2\}^2$ and the value of $\mu$ given in \eqref{mu=312again}. Following the method leading to Proposition \ref{PropforC5} we arrive at 

\begin{proposition}
\label{PropforC6}
Uniformly for $X\geq 2$, one has the equality
\begin{equation*}
{\rm Heis}^{\eqref{C6}} (X) = 2^{-1} \cdot 3^{-7} \alpha_3\, H_2\,  X^{1/4}+O (X^{1/4} \LL^{-1}).
\end{equation*}
\end{proposition}

\subsection{\texorpdfstring{Study of ${\rm Heis}^{\eqref{C7}}(X)$}{Study of Heis(X), VI}}
In our final case $\mu (f, f',d)$ satisfies \eqref{mu=316}. The proof mimics what was done for ${\rm Heis}^{\eqref{C6}} (X)$ since we also have $\eta, \eta' \in \{1, 2\}$. To detect the last condition of \eqref{C7} we use the sums 
$$
\frac 13 \Bigl( 1 +j^2\cdot  \chi (f' (3)\cdot f + 
2 f(3) \cdot f')(3) +j\cdot   \chi (2f' (3)\cdot f + 
f(3) \cdot f')(3)\Bigr)
$$
and   
$$
\frac 13 \Bigl( 1 + j\cdot  \chi (f' (3)\cdot f + 
2 f(3) \cdot f')(3) + j^2 \cdot  \chi (2f' (3)\cdot f + 
f(3) \cdot f')(3)\Bigr)
$$
that we insert in the first product on the right--hand side of \eqref{tripleprod} by changing the product $\prod_{r\mid \Delta_0}$ to $\prod_{r \mid 3\Delta_0}$. Finally, we conclude that

\begin{proposition}
\label{PropforC7}
Uniformly for $X\geq 2$, one has the equality
\begin{equation*}
{\rm Heis}^{\eqref{C7}} (X) = 3^{-8} \alpha_3\,  H_2\,  X^{1/4}+O (X^{1/4} \LL^{-1}).
\end{equation*}
\end{proposition}

\end{document}